\documentclass{amsart}

\usepackage{hyperref}

\usepackage{%
	 amsmath, amsfonts, amssymb, amsthm, dsfont, mathrsfs, wasysym, graphics, graphicx, xcolor, url, hyperref, hypcap,
	 xargs, multicol, pdflscape, multirow, hvfloat, array, ae, aecompl, pifont, mathtools, a4wide, float, blkarray, overpic, nicefrac, stmaryrd, anyfontsize, yfonts, array, tabularx, xfrac}

\usepackage[new]{old-arrows}
\usepackage[T1]{fontenc}
\usepackage{bbm}
\usepackage[shortlabels, inline]{enumitem}
\usepackage[noabbrev,capitalise]{cleveref}
\usepackage[normalem]{ulem}
\usepackage{marginnote}

\hypersetup{colorlinks=true, citecolor=darkblue, linkcolor=darkblue}
\usepackage[all]{xy}
\usepackage{tikz}
\usepackage{tikz-cd}
\usepackage{tkz-graph}
\usetikzlibrary{trees, decorations, decorations.pathmorphing, decorations.markings, decorations.shapes, shapes, arrows, matrix, calc, fit, intersections, patterns, angles}

\usepackage{stackengine}

\definecolor{darkblue}{rgb}{0,0,0.7} %
\definecolor{green}{RGB}{57,181,74} %
\definecolor{violet}{RGB}{147,39,143} %
\newcommand{\darkblue}{\color{darkblue}} %

\usepackage{paralist}

\newtheorem{thmUniv}{Theorem}

\newtheorem{theorem}{Theorem}[section]
\newtheorem{corollary}[theorem]{Corollary}
\newtheorem{proposition}[theorem]{Proposition}
\newtheorem{lemma}[theorem]{Lemma}
\newtheorem{conjecture}[theorem]{Conjecture}
\newtheorem*{theorem*}{Theorem}%

\theoremstyle{definition}
\newtheorem{definition}[theorem]{Definition}

\newtheorem{example}[theorem]{Example}
\newtheorem{remark}[theorem]{Remark}
\newtheorem{question}[theorem]{Question}

\newcommand{\R}{\mathbb{R}} %
\renewcommand{\c}[1]{{\mathcal{#1}}} %
\renewcommand{\b}[1]{{\boldsymbol{#1}}} %
\newcommand{\scr}[1]{{\mathscr{#1}}} %

\renewcommand{\emptyset}{\varnothing} %
\renewcommand{\epsilon}{\varepsilon} %

\newcommand{\ssm}{\smallsetminus} %
\newcommand{\inner}[1]{\left\langle#1\right\rangle}

\DeclareMathOperator{\conv}{conv} %

\DeclareMathOperator{\interior}{int}
\DeclareMathOperator{\facetnumber}{\#_{\text{facets}}}
\DeclareMathOperator{\coatomnumber}{\#_{\text{co-atoms}}}
\DeclareMathOperator{\corank}{\text{co-rk}}
\DeclareMathOperator{\argmin}{argmin}
\DeclareMathOperator{\argmax}{argmax}
\DeclareMathOperator{\GPS}{\boxdot}

\newcommand{\ie}{\textit{i.e.}~} %
\newcommand{\eg}{\textit{e.g.}~} %
\newcommand{\defn}[1]{\textsl{\darkblue #1}} %
\newcommand{\mathdefn}[1]{{\darkblue #1}} %

\newcommand{\pol}[1][P]{\mathsf{#1}}
\newcommand{\polQ}{{\pol[Q]}}
\newcommand{\polF}{{\pol[F]}}
\newcommand{\polC}{{\pol[C]}}
\newcommand{\polR}{{\pol[R]}}
\newcommand{\polS}{{\pol[S]}}
\newcommand{\polT}{{\pol[T]}}
\newcommand{\DefoCone}[1][\pol]{\mathbb{DC}(#1)} %
\newcommand{\EDefoCone}[1][\pol]{\mathbb{DC}_E(#1)} %
\newcommand{\FDefoCone}[1][\pol]{\mathbb{DC}_F(#1)} %

\newcommand{\SC}[1][n]{\mathbb{SC}_{#1}} %
\newcommand{\GP}[1][n]{\mathbb{GP}_{#1}} %
\newcommand{\permuto}[1][n]{\pol[\Pi]_{#1}} %

\DeclareMathOperator{\lineal}{lineal} %

\DeclareMathOperator{\FL}{\scr L} %
\newcommand{\comref}{\wedge }%

\newcommandx{\biflower}[2][1=\polQ, 2=\pol]{\c F_{#2, #1}}
\newcommandx{\biflowerE}[3][1=\polQ, 2=\pol, 3=\b E]{\c F_{#2, #1, #3}}
\newcommand{\lift}[1]{\uparrow\!{#1}}
\newcommand{\proj}[1]{\downarrow\!{#1}}
\newcommand{\liftmap}{\phi^\uparrow}
\newcommand{\projmap}{\phi^\downarrow}
\newcommand{\isquotient}{\preceq}
\newcommandx{\quotientQP}[2][1=\pol,2=\polQ]{#2\isquotient #1}
\newcommandx{\quotientQPE}[3][1=\pol,2=\polQ,3=\b E]{#2\isquotient_{#3}#1}
\newcommandx{\genquotientQP}[2][1=\pol,2=\polQ]{#2\prec #1}
\newcommandx{\admisscone}[2][1=\pol, 2=\polQ]{\mathcal{C}(#1, #2)}
\newcommandx{\admissconeE}[3][1=\pol, 2=\polQ, 3=\b E]{\mathcal{C}(#1, #2, #3)}
\newcommandx{\Iadmisscone}[2][1=\pol, 2=\polQ]{\widetilde{\mathcal{C}}(#1, #2)}
\newcommandx{\GPsum}[2][1=\polQ, 2=\pol]{#1\GPS#2}
\newcommandx{\genericGPsum}[2][1=\polQ, 2=\pol]{#1\square#2}
\newcommandx{\liftGPsum}[2][1=\polQ, 2=\pol]{\lift\!(\GPsum[#1][#2])}
\newcommandx{\liftGPsumOLD}[2][1=\polQ, 2=\pol]{\psi(#1{\GPS}#2)}
\newcommandx{\liftGPsumLM}[2][1=\polQ, 2=\pol]{#1{\,\raisebox{-0.13cm}{\rotatebox[origin=l]{90}{$\mapsto$}}\,}#2}
\newcommandx{\GPsumLM}[4][1=\polQ, 2=\pol, 3=\lambda, 4 = \mu]{#1\GPS_{#3, #4}#2}
\newcommand{\deformed}{\trianglelefteq}

\newcommand{\normeq}{\sim}%

\usepackage{todonotes}

\title{Many rays of the submodular cone}

\author{Georg Loho}
\address[Georg Loho]{Freie Universit\"at Berlin \& University of Twente}
\email{georg.loho@math.fu-berlin.de}

\author{Arnau Padrol}
\address[Arnau Padrol]{Universitat de Barcelona \& Centre de Recerca Matemàtica}
\email{arnau.padrol@ub.edu}

\author{Germain Poullot}
\address[Germain Poullot]{Universität Osnabrück}
\email{germain.poullot@uni-osnabrueck.de} 

\thanks{The research of A.P. is partially supported by the Spanish project PID2022-137283NB-C21 of MCIN/AEI/10.13039/501100011033 / FEDER, UE, the Spanish--German project COMPOTE (AEI PCI2024-155081-2 \& DFG 541393733), the Severo Ochoa and María de Maeztu Program for Centers and Units of Excellence in R\&D (CEX2020-001084-M), the Departament de Recerca i Universitats de la Generalitat de Catalunya (2021 SGR 00697), and the French--Austrian project PAGCAP (ANR-21-CE48-0020 \& FWF I 5788). }

\begin{document}

\begin{abstract}
The study of the cone of submodular functions goes back to Jack Edmonds' seminal 1970 paper, which already highlighted the difficulty of characterizing its extreme rays. Since then, researchers from diverse fields have sought to characterize, enumerate, and bound the number of such rays. In this paper, we introduce an inductive construction that generates new rays of the submodular cone. This allows us to establish that the $n^\text{th}$ submodular cone has at least $2^{2^{n-2}}$ rays, which improves upon the lower bound obtained from Hien Q. Nguyen's 1986 characterization of indecomposable matroid polytopes by a factor of order $\sqrt{n^3}$ in the exponent.
\end{abstract}

\maketitle

\section{Introduction}

\subsection{Background}
A submodular function is a set function that encodes dependence among finitely many objects.  %
Their study goes back to the work by Choquet from an analytic point of view~\cite{Choquet:1955}, and to Edmonds' foundational work in combinatorial optimization~\cite{Edmonds-SubmodularMatroidsPolyhedra}.  
Submodular functions are fundamental mathematical objects, which explains their ubiquity across different research areas. In particular, they arise as important tools in game theory~\cite{Shapley:1971}, algebraic combinatorics~\cite{AguiarArdila:2023}, and machine learning~\cite{Bach:2013,bilmes-submodularity-in-ai-ml-2022}. 
They generalize matroid rank functions~\cite{Whitney:1935,Nakasawa:1935} and form the fundament of discrete convex analysis~\cite{Murota:2003}, a theory of discrete convex functions that has recently gained new momentum through its deep connections with Lorentzian polynomials~\cite{BraendenHuh:2020}.

Submodular functions can be broken down in indecomposable pieces: a submodular function~$f$ is \emph{irreducible} if for any representation $f = g + h$ with submodular functions $g$ and $h$, the latter differ only by a scalar multiple and addition of a modular function from $f$. Actually, the set of submodular functions forms a convex cone, closed under addition and multiplication by nonnegative scalars. The rays of this cone (once modular functions are quotiented) are given by the irreducible submodular functions. Each submodular funcion can be broken down as a nonnegative sum of these indecomposable pieces.
These form the building blocks to represent all submodular functions, and already Edmonds remarks in his seminal paper~\cite{Edmonds-SubmodularMatroidsPolyhedra} that their characterization is a difficult task.  

In the same paper, he establishes the intimate connection of submodular functions with their associated polytopes.
These are known under various names due to their occurence in diverse fields: \emph{deformed permutahedra}, \emph{submodular base polyhedra}, \emph{M-convex polytopes}, \emph{(discrete) polymatroids} or \emph{generalized $n$-permutahedra}.
Postnikov observed that these polytopes are exactly the deformations of the regular \emph{permutahedron}~\cite{Postnikov-PermutohedraAssociahedraAndBeyond}. 
We take this point of view to derive our results. 
This goes back to McMullen~\cite{McMullen-typeCone} who studied the \emph{deformation cone} associated with a polytope $\pol$; that is the set of those polytopes whose normal fan coarsens the normal fan of $\pol$. 
For the (regular) permutahedron, this normal fan is just the braid arrangement given by the hyperplanes of the form $x_i = x_j$, where the full-dimensional cones are in bijection with permutations. 
In the language of submodular functions, this exactly corresponds to the fact that they are characterized by the greedy algorithm, selecting an optimal solution associated with a permutation of the elements of the ground set.

In the language of deformation cones, the polytopes associated with the irreducible submodular functions are exactly the \emph{Minkowski indecomposable} polytopes forming the rays of the deformation cone. 
Deformation cones also play an important role in toric geometry where they arise under the name \emph{nef cones}, which encode the embeddings of the associated toric variety into projective space~\cite{CoxLittleSchenckToric}.
For the permutahedral variety, the nef cone is just the cone of submodular functions. 
The question for an understanding of the rays of this cone (and generalizations) has also been posed in this context~\cite{HuhPhD,ArdilaCastilloEurPostinkov2020-CoxeterSubmodular}. Finally, note that this question also arose in economics: submodular functions (and their important subclass the gross substitute valuations) are widely used as valuation functions for goods, \eg in auctions, which makes their encoding a question of practical importance~\cite{Leme:2017}. 

The quest for understanding these indecomposable pieces has inspired significant work in different communities.
The complete list of the $37$ rays of the submodular cone on $4$ elements was already computed by Shapley in 1971~\cite{Shapley:1971}. For $n=5$, 
there are $117\, 978$ rays \cite{studeny2000extreme,CsirmazCsirmaz-AttemptingTheImpossible}, and for $n = 6$, the number is estimated to lie between $10^{13}$ and $10^{20}$, with exact enumeration considered ``impossible'' \cite{CsirmazCsirmaz-AttemptingTheImpossible}.
Other works establish various criteria for irreducibility~\cite{Nguyen1986-SemimodularFunctions,GirlichHodingSchneidereitZaporozhets:1995,Kashiwabara:2000,StudenyKroupa:2016,GrabischKroupa:2019}, or provide explicit constructions of families~\cite{RosenmullerWeidner1973-ExtremeConvexSetFunctions,studeny2016basicfactsconcerningsupermodular,HaimanYao:2023,PadrolPoullot-NewCriteriaIndecomposability}, including the identification of connected matroids as rays~\cite{Nguyen1986-SemimodularFunctions}, which was the largest known family until now.

Although a complete explicit characterization is out of reach, current estimates can still be refined. In this work, we present a constructive inductive procedure to enumerate a large number of irreducible submodular functions, enabling us to derive novel lower bounds on their total number.

\subsection{Results}

As stated above, we study submodular functions via their associated polytopes. 
For a submodular function $f \colon 2^{[n]} \to \R$, we consider the associated \emph{deformed permutahedron}, which is the polytope $\{\b x \in \R^n ~;~ \forall S \subseteq [n],\,\sum_{i \in S}x_i \leq f(S) \text{ and } \sum_{i \in [n]}x_i = f([n]) \}$. 
We will mainly use the characterization of these polytopes in terms of their edge directions: a polytope is a deformed permutahedron if and only if all edge directions are of the form $\b e_i - \b e_j$.
The projection of a generalized permutahedron, dropping the last coordinate, gives rise to a \emph{generalized polymatroid}.
These are exactly those polytopes with edge directions of the form $\b e_i - \b e_j$ and $\b e_i$.
Hence, from a polytope point of view, generalized polymatroids and deformed permutahedra are basically the same thing. 
However, viewing generalized polymatroids as objects in their own right has a conceptual advantage that has already proven useful in the original work~\cite{FrankTardos:1988}. 

This stems from the crucial observation that a generalized polymatroid is uniquely given by a pair of faces that are deformed permutahedra, namely the minimal and the maximal face with respect to taking the coordinate sum as objective function (see Theorem~\ref{thm:bijections}).
These are the \emph{top} and \emph{bottom} faces. 

For the other direction, to associate a generalized polymatroid to a pair $(\pol,\polQ)$ of deformed permutahedra, the latter two have to lie in space in a compatible way.
Luckily, by suitably translating the two polytopes, this can actually always be guaranteed. 
By lifting the resulting generalized polymatroid one dimension higher, one can construct a new deformed permutahedron.
This is the basis of our construction.

Even if the two deformed permutahedra are \emph{indecomposable}, the resulting polytope is not necessarily indecomposable.
To get an indecomposable deformed permutahedron, we need to dilate $\pol$ and $\polQ$, and to translate them suitably (Proposition~\ref{prop:SeedIsRay}).
To find these translated dilates, the crucial insight is that the set of generalized polymatroids whose top and bottom faces are deformations of a given pair of deformed permutahedra, the \emph{admissibility cone}, is isomorphic to a deformation cone (\Cref{thm:AdmissibilityConeIsDeformationCone}).
The rays of the admissibility cone give us our desired higher dimensional indecomposable deformed permutahedra.
However, this does not always allow for our recursive construction (see \Cref{exm:FamillyOfPairwiseFertileN4}) -- we also need to chose the original pairs carefully (\Cref{def:FertilePair}).

Our construction ensures that we can actually iterate the process and still get rays from new pairs. 
Let $\SC = \{f \colon 2^{[n]} \to \R ~;~ f \text{ submodular}\}$. 

\begin{thmUniv}[{\Cref{cor:CollectionOfRays}}]
Given a suitably chosen subset of the rays of the submodular cone $\SC$, each pair gives rise to (at least one) ray of the submodular cone $\SC[n+1]$. 
Applying the same construction to the resulting subset of rays of $\SC[n+1]$ yields rays of $\SC[n+2]$, and so on.  
\end{thmUniv}

Our construction allows us to give an improved lower bound on the number~$t_n$ of rays of the submodular cone~$\SC$ (\cref{cor:ExplicitLowerBoundNumberOfRays}), which we complete with an asymptotic upper bound (\cref{prop:upperbound}). 

\begin{thmUniv}[{\Cref{thm:UpperLowerBoundsTnSn}}]
For $n$ sufficiently large:
$$2^{2^{n-2}} \,\leq\, t_n \,\leq\, n^{2^n}$$

\end{thmUniv}

This improves the previously known bounds reported by Haiman and Yao in~\cite{HaimanYao:2023}, namely: $2^{c\cdot2^n\cdot n^{-3/2}}\leq t_n\leq n^{2^{n+1}}$ for some explicit $c > 0$. In particular, the best lower bound was given by the number 
 the number $m_n$ of rays coming from (connected) matroids~\cite{Nguyen1986-SemimodularFunctions}. 
By~\cite{BansalPendavinghVanDerPol2015-NumberMatroids}, $\log_2\log_2 m_n \leq n - \frac{3}{2}\log_2 n + \frac{1}{2}\log_2 \frac{2}{\pi} + 1 + o(1) \leq n-2$, which 
shows that our construction goes beyond the number of rays coming from (connected) matroids.
To our knowledge, the best upper bound was about $n^{2^{n+1}}$, from~\cite[Theorem 5.1]{HaimanYao:2023}. We use McMullen's Upper Bound Theorem~\cite{McMullen-MaximumNumberOfFaces} to reduce this upper bound to roughly its square root.

Our construction allows to associate each constructed ray explicitly to a sequence of combinations of the elements in the starting collection. 
Extensions of our technique might also help to give even better lower bounds closer to the upper bound by combining it with other constructions. 

In \Cref{sec:equalitysets} we show that our inductive approach is not only useful fur studying rays, but aslo higher-dimensional faces of submodular cones. Since this extension is not necessary for our main result, we have postponed it to an appendix.

\subsection*{Acknowledgments}
We thank Federico Ardila and Christian Haase for clarifying the relation between the submodular cone and toric geometry.
We also thank Yuan Yao for pointing us at his recent paper on bounds on the number of rays of the submodular cone, during FPSAC 2025 in Sapporo.
We are thankful to the MFO for the conference ``Geometric, Algebraic, and Topological Combinatorics'' where we started collaborating on the present project, in December 2023. Finally, G.P. wants to deeply thank Lionel Pournin and Jes\'us De Loera for challenging him to improve (respectively) the lower and the upper bound, in October 2023.

\section{Generalized-polymatroid sum and admissibility cone}\label{sec:AdmissCone}

This section introduces the tools for our construction. We use the language of generalized polymatroids, because even though they are equivalent to deformed permutahedra, their presentation highlights the inductive structure that we will exploit. Our main contribution in this section is the definition of admissibility cones in \cref{ssec:AdmissCones}, and their interpretation as deformation cones. Most of the results that we present prior to the introduction of admissibility cones are known facts about generalized polymatroids \cite{Murota:2003,Fujishige:2005,Frank:2011}. 
However, since they are central for our applications, we provide some proofs for completeness. 
Our approach is more polytopal than in previous presentations, and we believe it may be valuable to readers with a background in polyhedral geometry rather than optimization.

\subsection{Polytopes, fans, and deformations}

A \defn{polyhedron} $\pol$ is the intersection of finitely many (closed) affine half-spaces of $\R^d$. When a polyhedron is bounded, it is called a \defn{polytope}, and can be alternatively described as  the convex hull of finitely many points in~$\R^d$.
For $\b c\in \R^d$, its \defn{face in direction $\b c$} is $\mathdefn{\pol^{\b c}} \coloneqq \{\b x\in \pol ~;~ \inner{\b x, \b c} = \max_{\b y\in \pol}\inner{\b y, \b c}\}$, where $\inner{\cdot, \cdot}$ denotes the usual scalar product of $\R^d$.
The \defn{dimension} of a face $\polF$ of $\pol$ is the dimension of the smallest affine sub-space containing~$\polF$.
By convention, the empty set $\emptyset$ is a face of $\pol$ of dimension~$-1$.

 The \defn{face lattice} of $\pol$ is the partially ordered set $\mathdefn{\FL(\pol)}$ whose elements are the faces of $\pol$, ordered by inclusion. %
 Two polytopes are \defn{combinatorially equivalent} if their face lattices are isomorphic.

A convex cone is a set closed under positive linear combinations. A \defn{polyhedral cone} $\polC$ is the set of positive linear combinations of a finite set of vectors, or equivalently, the intersection of finitely many (closed) linear half-spaces.
We denote by $\mathdefn{\interior(\polC)}$ the interior of a cone $\polC$. The \defn{lineality} of a cone~$\polC$ is the largest linear space it contains, and is denoted~$\lineal(\polC)$.

A \defn{fan} $\c F$ is a collection of polyhedral cones in $\R^d$ closed under taking faces and intersections. %
A fan $\c G$ \defn{coarsens} a fan $\c F$, denoted $\c G \deformed\c F$, if every cone of $\c G$ is a union of cones of~$\c F$.
The \defn{common refinement} of two fans $\c F$ and $\c G$ in $\R^d$ is the fan $\mathdefn{\c F \comref \c G}$ formed by the cones $\polC\cap\polC'$ for all $\polC\in \c F$ and $\polC'\in \c G$.

For a face $\polF$ of a polytope $\pol$, its \defn{(closed) normal cone} is $\mathdefn{\c N_{\pol}(\polF)} \coloneqq \{\b c\in \R^d ~;~ \pol^{\b c} \subseteq \polF\}$.
The \defn{normal fan} of~$\pol$ is the collection of cones $\mathdefn{\c N_{\pol}} \coloneqq \bigl\{\c N_{\pol}(\polF) ~;~ \polF \text{ face of } \pol\bigr\}$.
Two polytopes $\pol,\pol'$ are \defn{normally equivalent} if they have the same normal fan, which we denote by $\pol\normeq \pol'$.
Normal equivalence implies combinatorial equivalence. %

The \defn{Minkowski sum} of $\pol$ and $\polQ$ is the polytope $\mathdefn{\pol+\polQ} \coloneqq \{\b p+\b q ~;~ \b p\in \pol,\, \b q\in \polQ\}$. Its normal fan is the common refinement $\c N_{\pol}\comref\c N_{\polQ}$, see \cite[Proposition 7.12]{Ziegler-polytopes}. A polytope $\pol$ is \defn{indecomposable} if $\pol=\polQ+\polR$ imples that $\polQ$ and $\polR$ are translated dilates of $\pol$.

We will be interested in the set of deformations of a polytope, which we now define. Polytope deformations have been studied in several contexts under different names, which explains the existence of these different characterizations (\eg in~\cite{Shephard1963,McMullen-typeCone,Meyer,Postnikov-PermutohedraAssociahedraAndBeyond}). 

\begin{definition}\label{def:deformation}
A polytope $\polQ$ is a \defn{deformation} of a polytope $\pol$, which we denote \defn{$\polQ\deformed\pol$}, if it fulfills any of the following equivalent properties:
\begin{compactenum}[(i)]
    \item $\polQ$ is a weak Minkowski summand of $\pol$, \ie there exists a real $\lambda > 0$ and a polytope~$\polR$ such that $\polQ + \polR = \lambda \pol$;
    \item the normal fan of $\polQ$ coarsens the normal fan of $\pol$, \ie $\c N_{\polQ} \deformed\c N_{\pol}$; 
    \item $\polQ$ is obtained from $\pol$ by moving the vertices so that the directions of all edges are preserved, \ie there is a (non-necessarily injective) map $\b q_i\mapsto \b p_i$ from the vertices of $\polQ$ to the vertices of $\pol$ such that for every edge $\b p_i\b p_j$ of $\pol$, we have $\b q_i-\b q_j=\lambda_{ij}(\b p_i-\b p_j)$ for some $\lambda_{ij}\ge0$.
\end{compactenum}
\end{definition}

A proof of the equivalence of these conditions can be found in \cite[Theorem 15.3]{PostnikovReinerWilliams:2008}. 

Note that if two polytopes $\polQ$ and $\polR$ are deformations of $\pol$, so are $\polQ+\polR$ and $\lambda\polQ$ for any $\lambda > 0$. 
Consequently, the set of deformations of $\pol$ forms a cone in the vector space of virtual polytopes as defined in~\cite{PukhlikovKhovanskii} (\ie it is closed under taking sums and positive dilations of polytopes). Distinct characterizations of deformations naturally give rise to distinct parameterizations of this cone.
See \Cref{fig:SC3} for an illustration of a deformation cone.

\begin{definition}\label{def:defocone}
Let $\pol\subseteq\R^d$ be a polytope whose sets of vertices, facets and edges are respectively denoted $V$, $F$, and $E$, and whose inequality description is $\pol= \{\b x\in\R^d~;~ h_{\pol[f]}(\b x)\leq b_{\pol[f]}\text{ for } {\pol[f]}\in F\}$.

\begin{compactenum}[(i)]
\item Its \defn{deformation cone} is the set of polytopes:
$$\mathdefn{\DefoCone} \coloneqq \{\polQ \subseteq\R^d~;~ \polQ\deformed\pol\}.$$
\item Its \defn{facet deformation cone} is the set $\mathdefn{\FDefoCone}$
 of ``right-hand-sides'' $\b \beta\in \R^F$ that satisfy: 
$$ \c N_{\pol({\b \beta})} \deformed\c N_{\pol}$$
where $\pol({\b \beta})\coloneqq \{\b x\in\R^d~;~ h_{\pol[f]}(\b x)\leq \beta_{\pol[f]}\text{ for } {\pol[f]}\in F$\}.
\item Its \defn{edge deformation cone} is the set $\mathdefn{\EDefoCone}$
 of edge-lengths $\b \lambda\in \R^E$ that satisfy:
$$
\lambda_{\b v_1\b v_2}(\b v_2-\b v_{1}) +\lambda_{\b v_2\b v_3}(\b v_3-\b v_2) +\cdots+\lambda_{\b v_k\b v_1}(\b v_1-\b v_k) = \b 0
$$ for any $2$-dimensional face of $\pol$ with cyclically ordered vertices $\b v_1,\dots,\b v_k$.
\end{compactenum}
\end{definition}

Note that, for any $\b t\in \R^d$, we have $\c N_{\polQ + \b t} = \c N_\polQ$ (where $\polQ + \b t$ denote the translation of $\polQ$ by the vector $\b t$).
Thus, for any $\polQ\in \DefoCone$, we have $\{\polQ + \b t ~;~\b t\in \R^d\} \subseteq\DefoCone$.
Indeed, for $\pol\subseteq\R^d$ the cone $\DefoCone$ has a lineality space of dimension $d$ (which might be $\geq \dim \pol$) induced by the space of translations. Similarly, if $\beta_{\pol[f]}\in \FDefoCone$, then $\beta_{\pol[f]} + h_{\pol[f]}(\b t)\in \FDefoCone$ for any $\b t\in \R^d$.
The cone $\FDefoCone$ has a lineality space of dimension $d$ induced by translations too.

As it can be seen for example in \cite[Thm. 15.5]{PostnikovReinerWilliams:2008}, modulo their lineality, the cones $\DefoCone$, $\EDefoCone$ and $\FDefoCone$ are linearly isomorphic. To avoid confusion, we will mainly work with the presentation~$\DefoCone$ of the cone of deformations. Nevertheless,  the other parameterizations will implicitly appear in some definitions and proofs.

\begin{lemma}
For any polytope $\pol\subseteq\R^d$, we have:
$$\DefoCone/\lineal(\DefoCone)~\simeq~\EDefoCone ~\simeq~\FDefoCone/\lineal(\FDefoCone)$$
where $\lineal(\DefoCone) \simeq \lineal(\FDefoCone) \simeq \R^d$.
\end{lemma}

Note that, from these equivalent parameterizations, it is clear that $\DefoCone$ is a polyhedral cone. 
Each of its \emph{faces} is the deformation cone $\DefoCone[\polQ]$ of a deformation $\polQ\deformed\pol$, and the interior of this face contains all the polytopes normally equivalent to~$\polQ$ (in particular, $\interior(\DefoCone)$ contains all polytopes normally equivalent to $\pol$).
With all this in mind, for a polytope $\pol\subseteq \R^d$, we will slightly abuse notation and say that a polytope $\polQ$ ``is a face'' of $\DefoCone$, meaning that $\DefoCone[\polQ]$ is a face of $\DefoCone$; similarly, we say that $\polQ$ ``is a ray'' of $\DefoCone$ if moreover $\dim\DefoCone[\polQ] - d = 1$. The rays of $\DefoCone$ are in correspondence with the indecomposable deformations of $\pol$.

\subsection{Deformed permutahedra and generalized polymatroids}

The \defn{$n$-permutahedron $\permuto$}$\subseteq\R^n$ is the convex hull of the $n!$ points $\bigl(\sigma(1),\dots,\sigma(n)\bigr)\in \R^n$ ranging over all the permutations $\sigma\in \c S_n$. Its normal fan is the \defn{braid fan} whose maximal cones are $\pol[C]_\sigma \coloneqq \{\b c \in \R^n ~;~ c_{\sigma(1)} \leq c_{\sigma(2)} \leq \dots \leq c_{\sigma(n)}\}$ for $\sigma\in \c S_n$, see \Cref{fig:VariousGPsums} (Top left). 

\begin{definition}\label{def:deformedpermutahedra}
A polytope $\pol$ is a \defn{deformed $n$-permutahedron} 
if it satisfies any of the following equivalent conditions:
\begin{compactenum}[(i)]
   \item $\pol\deformed\permuto$;
    \item\label{it:dpsumodular}  
    $\pol = \{\b x \in \R^d~;~\forall S\subsetneq [n],\,\sum_{i \in S}x_i \leq f(S) ~;~ \sum_{i \in [n]}x_i = f([n]) \}$
   for a \defn{submodular} set function $f: 2^{[n]}\to \R$, \ie for all $A, B\subseteq [n]$, we have $f(A) + f(B) \,\geq\, f(A \cup B) + f(A \cap B)$; or
   \item\label{it:dpedges}  all its edges are\footnote{By ``being in direction of $\b z$'', we mean that if $[\b u, \b v]$ is an edge of $\pol$, then there exists $\lambda\in \R^*$ with $\b v - \b u = \lambda\b z$.} in direction $\b e_i - \b e_j$ for some $i, j\in [n]$.
\end{compactenum}
\end{definition}
For details on the edge and facet presentations given here, see for example \cite[Section 3]{Fujishige:2005} or \cite{Postnikov-PermutohedraAssociahedraAndBeyond}. 
The deformation cone of the permutahedron $\mathdefn{\SC} \coloneqq \DefoCone[\permuto]$ is called the \defn{submodular cone} because the facet-deformation cone is in bijection with the cone of submodular set functions, by \cref{def:deformedpermutahedra}\eqref{it:dpsumodular}, see \Cref{fig:SC3} for an illustration of $\SC[3]$.

The normal fan $\c N_{\pol}$ of a deformed $n$-permutahedron $\pol$ is a coarsening of the braid fan, and therefore its maximal cones are unions of cones $\pol[C]_\sigma$ for certain permutations $\sigma\in \c S_n$.
This implies that for $\sigma\in \c S_n$ and any $\b c\in \interior(\pol[C]_\sigma)$, the face $\pol^{\b c}$ is a vertex of $\pol$ which does not depend on the actual choice of $\b c\in \interior(\pol[C]_\sigma)$.
Hence, we denote by \defn{$\pol^\sigma$} the vertex $\pol^{\b c}$ for any $\b c\in \interior(\pol[C]_{\sigma})$.

As per usual, we denote the $i^{\text{th}}$ coordinate of a vector $\b x\in \R^n$ (\ie $\inner{\b x, \b e_i}$) by \defn{$x_i$}.
We write~$\pol^\sigma_i$ for the $i^{\text{th}}$ coordinate of the vertex $\pol^\sigma$.
Note that, for a deformed $n$-permutahedron $\pol$, the quantity $\sum_{i=1}^n x_i$ is constant over $\b x\in \pol$, \ie $\pol$ lies in a hyperplane orthogonal to $\mathdefn{\b 1} \coloneqq (1, \dots, 1)$.

\begin{definition}
For a deformed $n$-permutahedron $\pol\in \SC$, the faces $\pol^{-\b e_n}$ and  $\pol^{+\b e_n}$ are respectively called the \defn{top} and \defn{bottom faces}\footnote{We choose this notation for consistency with the upcoming \cref{def:topbotGP}.}.
\end{definition}

The importance of this pair of antipodal faces is highlighted by the following result:

\begin{lemma}\label{lem:topbotgenperm}
A deformed $n$-permutahedron $\pol\in \SC$ is uniquely determined (among deformed $n$-permutahedra) by the pair of deformed permutahedra $(\pol^{+\b e_n},\, \pol^{-\b e_n})$.
\end{lemma}

\begin{proof}
Let $\pol$ be a deformed $n$-permutahedron with bottom face $\polR \coloneqq \pol^{+\b e_n}$ and top face $\polQ \coloneqq \pol^{-\b e_n}$.
As edges of $\polR$ and $\polQ$ are edges of $\pol$, \Cref{def:deformedpermutahedra}\eqref{it:dpedges} implies that $\polR$ and $\polQ$ are deformed permutahedra.
We denote $\b e_X = \sum_{i\in X} \b e_i$ for $X\subseteq [n]$.

By \cref{def:deformedpermutahedra} we have $\pol = \{\b x\in \R^n ~;~ \forall X\subseteq [n],\, \inner{\b x, \b e_X} \leq f(X), \inner{\b x, \b e_{[n]}} = f([n])\}$ for a certain submodular function $f : 2^{[n]} \to \R$. We will prove that $f$ is entirely determined by $\polQ$ and $\polR$.
Note that, by definition, $f(X) = \max_{\sigma\in \c S_n} \inner{\pol^\sigma, \b e_X}$.
Moreover, for any $\tau\in \c S_n$, if $\tau([n-|X|+1,\, n]) = X$, then we have $\b e_X\in \polC_\tau$, and consequently $f(X) = \inner{\pol^\tau, \b e_X}$.

Consider $X$ such that $n\in X$, and let $\sigma\in \c S_n$ be such that $\sigma(n) = n$, and $\sigma([n-|X|+1,\, n]) = X$.
We have $\pol^\sigma\in \pol^{+\b e_n} = \polR$, which implies $f(X) = \inner{\pol^\sigma, \b e_X} = \inner{\polR^\sigma, \b e_X}$.
Conversely, if $n\notin X$, then, taking $\sigma\in \c S_n$ such that $\sigma(1) = n$ and $\sigma([n-|X|+1,\, n]) = X$, we get $\pol^\sigma \in \pol^{-\b e_n} = \polQ$ and consequently $f(X) = \inner{\pol^\sigma, \b e_X} = \inner{\polQ^\sigma, \b e_X}$.
Hence, $f$ is entirely determined by $\polR$ and $\polQ$.
\end{proof}

As we precise in \cref{thm:bijections}, this leads to a bijection between deformed $(n+1)$-permutahedra and ``compatible'' pairs of deformed $n$-permutahedra. This correspondence is best illustrated via so-called generalized $n$-polymatroids. 
Even though they are in bijection with deformed $(n+1)$-permutahedra (\cref{prop:LiftedGPsumProperties}), their presentation in~$\R^n$ (instead of $\R^{n+1}$) emphasizes the above interpretation, and they can be thought as an intermediate step in the construction of one deformed $(n+1)$-permutahedron out of two deformed $n$-permutahedra.

We start defining generalized polymatroids as a generalization of \cref{def:deformedpermutahedra}.

\begin{definition}\label{def:genpolymatroid}
A polytope $\polR\in \R^n$ is a \defn{generalized $n$-polymatroid}
if it satisfies any of the following equivalent conditions:
\begin{compactenum}[(i)]  
    \item $\pol \deformed (\permuto + [0, 1]^n)$;
    \item $\pol=\{\b x\in \R^d ~;~\forall S\subsetneq [n],\, b(S)\leq \sum_{i\in S} x_i\leq p(S)\}$
  for set functions $p,b: 2^{[n]}\to \R$ with $b$ supermodular (\ie $-b$ submodular), $p$ submodular , $p(\emptyset) = b(\emptyset) = 0$ and such that 
  $$p(S) - b(T ) \ge p(S\setminus T ) - b(T \setminus S)$$ 
  holds for all $S, T \subseteq [n]$; or
   \item all its edges are in direction $\b e_i - \b e_j$ or $\b e_i$ for some $i, j\in [n]$ with $i\ne j$.
\end{compactenum}
\end{definition}
The equivalence of these  presentations can be found for example in \cite[Theorem~3]{FrankKiralyPapPritchard}.
We now explain the equivalence between generalized $n$-polymatroids and deformed $(n+1)$-permutahedra.

\begin{definition}\label{def:LiftGPsum}
The \defn{lifting map} is the map $\mathdefn{\liftmap} :\R^n\to \R^{n+1}$ given by $\b x\mapsto \bigl(\b x,\, -\sum_{i=1}^n x_i\bigr)$, and the \defn{projection map} is the map $\mathdefn{\projmap} :\R^{n+1}\to \R^{n}$ given by $(x_1, \dots, x_n, x_{n+1}) \mapsto (x_1, \dots, x_n)$.
If $\pol\subset\R^n$ and $\polR\subset\R^{n+1}$ are polytopes, then the \defn{lift} of $\pol$ is $\mathdefn{\lift\pol} \coloneqq \liftmap(\pol)\subset \R^{n+1}$, and the \defn{projection} of $\polR$ is  $\mathdefn{\proj\polR} \coloneqq \projmap(\pol)\subset \R^{n}$.
\end{definition}

\begin{proposition}\label{prop:LiftedGPsumProperties}
For a generalized $n$-polymatroid $\pol[S]$, its lift $\lift{\!\pol[S]}$ is a deformed $(n+1)$-permutahedron.
Conversely, if $\polR$ is a deformed $(n+1)$-permutahedron, then its projection $\proj{\!\polR}$ is a generalized $n$-polymatroid.
One has $\proj{\!(\lift{\pol[S]})}=\pol[S]$, and $\lift{\!(\proj{\polR})} = \polR - \left(\sum_{i=1}^{n+1} x_i\right)\b e_{n+1}$ for any $\b x\in \polR$.
\end{proposition}

\begin{proof}
This is basically 
\cite[Theorem 14.2.2 and 14.2.5]{Frank:2011}, we give here a self-contained proof.

If $\polR$ is a deformed $(n+1)$-permutahedron, then all its edges are in direction $\b e_i-\b e_j$ for some $i, j\in [n+1]$ with $i\ne j$.
As the edges of $\proj{\polR}$ are among the projections of the edges of $\polR$, we get that the edges of $\proj{\polR}$ are in directions $\projmap(\b e_i - \b e_{n+1}) = \b e_i$ for some $i\in [n]$, or in direction $\projmap(\b e_i - \b e_j) = \b e_i - \b e_j$ for some $i,j\in [n]$ with $i\ne j$.
Hence, $\proj{\polR}$ is a generalized $n$-polymatroid.

Reciprocally, an edge of a generalized $n$-polymatroid $\polS$ is either in direction $\b e_i$ for some $i\in [n]$, and we have $\liftmap(\b e_i) = \b e_i - \b e_{n+1}$; or it is in direction $\b e_i - \b e_j$ for some $i, j\in [n]$ with $i\ne j$, and we have $\liftmap(\b e_i - \b e_j) = \b e_i - \b e_j$.
Hence, $\lift{\polS}$ is a deformed $(n+1)$-permutahedron.

The equalities $\proj{(\lift{\pol[S]})}=\pol[S]$, and that $\lift{(\proj{\polR})} = \polR - \left(\sum_{i=1}^{n+1} x_i\right)\b e_{n+1}$ for any $\b x\in \R$, are immediate to check (recall that $\sum_{i=1}^{n+1} x_i$ is independent of the choice of $\b x\in \polR$).
\end{proof}

\subsection{Quotients and GP-sums}
We now proceed to formalize a converse statement to \cref{lem:topbotgenperm}, that is to understand which pairs of deformed $n$-permutahedra arise as top and bottom faces of a generalized $(n+1)$-permutahedron, with the help of the language of generalized polymatroids.

\begin{definition}\label{def:topbotGP}
The \defn{top face} of a generalized $n$-polymatroid $\polR$ is the polytope $\polR^{+\b 1}$, while its \defn{bottom face} is the polytope $\polR^{-\b 1}$.
\end{definition}

The property characterizing the pairs appearing as top/bottom faces is given by the following notion of \defn{forming a quotient}.
The name ``quotient'' is motivated by the fact that the relation encodes quotients of linear spaces arising from linear maps in the case of realizable (poly)matroids.

\begin{definition}
Two deformed $n$-permutahedra $\pol$ and $\polQ$ \defn{form a quotient}, denoted \defn{$\quotientQP$}, if for all permutations $\sigma\in \c S_n$ and all indices $i\in [n]$, we have $\polQ^\sigma_i \leq \pol^\sigma_i$. %
If all the inequalities defining the quotient are strict, \ie $\polQ^\sigma_i < \pol^\sigma_i$ for all $i\in [n]$ and $\sigma \in \c S_n$, then we say that the quotient is \defn{generic}, and we denote it by \defn{$\genquotientQP$}.
\end{definition}

\begin{remark}\label{rmk:GenericGPsum}
	For $\pol, \polQ\in \SC$, there is a translation $\polQ'$ of $\polQ$ such that $\polQ'$ and $\pol$ forms a generic quotient, \ie $\genquotientQP[\pol][\polQ']$. Indeed, if $\polQ' = \polQ + \b t$ for $t_i = \min_{\sigma\in \c S_n}(\pol^\sigma_i - \polQ^\sigma_i) - 1$, then it is straightforward to see that $\polQ'$ is a generic quotient of $\pol$.
\end{remark}

The \defn{GP-sum} is the operation that allows us to construct generalized $n$-polymatroids from pairs of deformed $n$-permutahedra forming a quotient.

\begin{definition}\label{def:GPsum}
Given two polytopes $\pol, \polQ$ in $\R^d$, 
their \defn{GP-sum} (standing for ``generalized polymatroid sum'') is the polytope:
$$\mathdefn{\GPsum} \coloneqq (\polQ + \R_+^d) \,\cap\, (\pol + \R_-^d),$$
see \Cref{fig:GPsum} for an example.
\end{definition}

\begin{figure}
\centering
\includegraphics[width=0.99\linewidth]{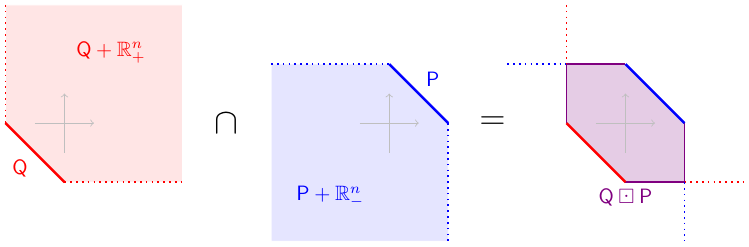}
\caption[Construction of the GP-sum $\GPsum$]{The construction of $\GPsum$: for $\pol$ and $\polQ$ two segments, the corresponding $\textcolor{blue}{\pol+\R_-^n}$ and $\textcolor{red}{\polQ+\R_+^n}$ are unbounded polyhedra.
Computing their intersection yields the GP-sum $\textcolor{violet}{\GPsum}$, whose top and bottom faces are $\pol$ and $\polQ$ respectively.
Note that $\GPsum$ is a polytope, and that the combinatorics of $\GPsum$ heavily depend on the exact coordinates of $\pol$ and $\polQ$, not on their sole combinatorics.}
\label{fig:GPsum}
\end{figure}

\begin{remark}
Even if it is described as the intersection of two unbounded polyhedra, note that the GP-sum $\GPsum$ of two polytopes $\pol$ and $\polQ$ is always a bounded polytope, because the $i$\textsuperscript{th} coordinate of its points is bounded between the minimum and maximum of the $i$\textsuperscript{th} coordinates of points in $\polQ$ and $\pol$, respectively, \ie we have: $\GPsum \subseteq \prod_{i=1}^n [\min_{\b y\in \polQ} y_i,\, \max_{\b x\in\pol} x_i]$.

\end{remark}

\begin{lemma}\label{lem:polymatroidtopbotquotient}
If $\polR$ is a generalized $n$-polymatroid, then $\polR^{-\b 1}$ and $\polR^{+\b 1}$ are deformed $n$-permutahedra forming a quotient.
\end{lemma}

\begin{proof}
Let $\polR$ be a generalized $n$-polymatroid, with $\pol \coloneqq \polR^{+\b 1}$ and $\polQ \coloneqq\polR^{-\b 1}$.
As $\pol$ maximizes $\b 1$, the polytope $\pol$ has no edges in direction $\b e_i$, hence, as the direction of the edges of $\polR$ are $\b e_i - \b e_j$ for some $i\ne j$ and $\b e_i$ for some $i\in [n]$, the polytope $\pol$ only has edges in direction $\b e_i - \b e_j$ for some $i\ne j$.
Thus, $\pol$ is a deformed $n$-permutahedron.
The same holds for $\polQ$.

To prove that they form a quotient, consider a generic vector $\b c\in\polC_\sigma$, which by definition lies in the normal cones of the vertices $\pol^\sigma$ and $\polQ^\sigma$.
Now, for $\lambda\in\R$, consider the vector $\b c_{\lambda} = \b c + \lambda \b 1$.
For $\lambda\to -\infty$ small enough, $\polR^{\b c_{\lambda}} = \polQ^{\sigma}$, while for $\lambda\to +\infty$ big enough $\polR^{\b c_{\lambda}} = \pol^{\sigma}$.
As $\lambda$ goes from $-\infty$ to $+\infty$, the vertices $\polR^{\b c_{\lambda}}$ describe a monotone edge path from $\polQ^{\sigma}$ to $\pol^{\sigma}$ (such paths are called \emph{parametric simplex paths} in linear programming).
The edges in this path are strictly increasing in the direction $\b 1$. Hence, they must be in direction $\b e_j$ for some $j$, thus increasing the $j^{\text{th}}$ coordinate.
This shows that  $\polQ^\sigma_i \leq \pol^\sigma_i$ for all $\sigma\in \c S_n$ and $i\in [n]$, that is, that $\pol$ and $\polQ$ form a quotient.
\end{proof}

Furthermore, the following auxiliary statement can be found, e.g., in \cite[Theorem~6.1]{BrandenburgLohoSmith2024-QuotientsMconvexSets}. For brevity, we omit an alternative proof and instead refer to their 
proof. 

\begin{lemma}\label{lem:GP-sumpolymatroid}
If $\pol$ and $\polQ$ are deformed $n$-permutahedra forming a quotient, then their GP-sum $\GPsum$ is a generalized polymatroid.
\end{lemma}

The following unicity result is a reformulation of \cite[Theorem 14.2.8]{Frank:2011}.

\begin{lemma}\label{lem:GeneralizedPolymatroisUniquelyDetermined}
If $\polR$ is a generalized $n$-polymatroid, then $\polR = \GPsum[(\polR^{-\b 1})][(\polR^{+\b 1})]$. In particular, a generalized polymatroid is uniquely determined by its top and bottom faces.
\end{lemma}

\begin{proof}
The lifting map sends the top face of a generalized polymatroid $\polR$ to the top face of the deformed permutahedron $\lift\!\polR$, and the bottom face of $\polR$ to the bottom face of $\lift\!\polR$.
By \cref{lem:topbotgenperm}, $\polR$ is uniquely determined by $\polR^{-\b 1}$ and $\polR^{+\b 1}$. Now, $\GPsum[(\polR^{-\b 1})][(\polR^{+\b 1})]$ is a generalized polymatroid with the same top and bottom faces as $\polR$ by \cref{lem:polymatroidtopbotquotient,lem:GP-sumpolymatroid}. Therefore: $\polR = \GPsum[(\polR^{-\b 1})][(\polR^{+\b 1})]$.
\end{proof}

The combination of these results provides the equivalence between generalized $n$-polymatroids and quotients of deformed $n$-permutahedra  (see also \cite{Frank:2011}, and \cite[Theorem 1.2 (2) \& (4)]{BrandenburgLohoSmith2024-QuotientsMconvexSets}).

\begin{figure}[t]
\centering
\includegraphics[width=0.9\linewidth]{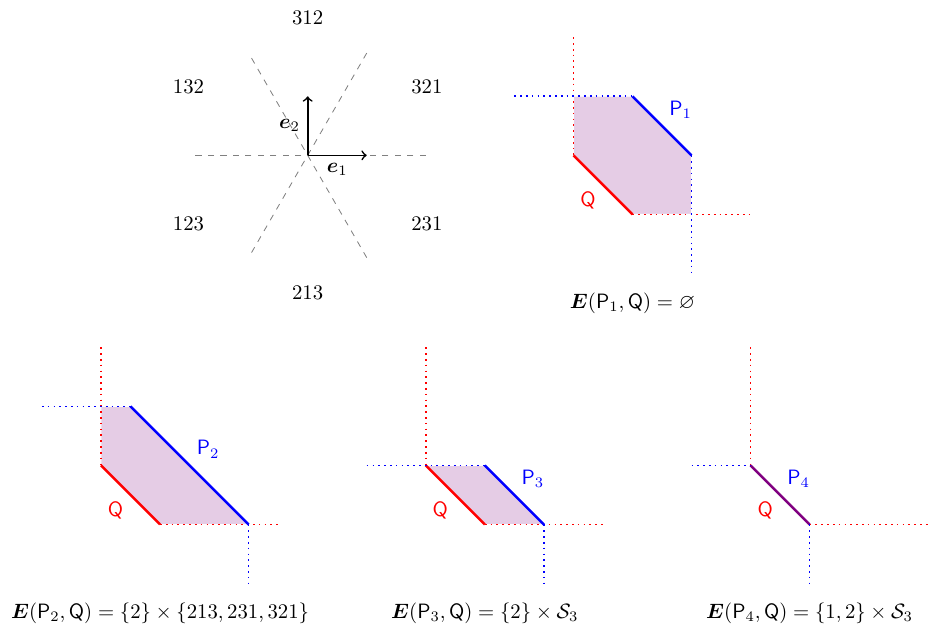}
\caption[Non-normally-equivalent GP-sums $\GPsum$ arising for normally equivalents $\pol_1, \dots, \pol_4$]{The notation $\b E(\pol, \polQ)$ will appear relevant in \Cref{sec:equalitysets}. (Top left) The braid fan of dimension 3 intersected with the plane spanned by $(\b e_1, \b e_2)$: each maximal chamber $\polC_\sigma$ is decorated with the corresponding permutation $\sigma\in \c S_3$.
(Top right \& Bottom) Various $\GPsum[\polQ][\pol_i]$ for normally equivalent $\pol_1,\dots, \pol_4$.
The normal fan of each $\GPsum[\polQ][\pol_i]$ is different.}
\label{fig:VariousGPsums}
\end{figure}

\begin{theorem}\label{thm:bijections}
The following three sets are in bijection:
\begin{compactenum}[(i)]
 \item\label{it:quotients} pairs of deformed $n$-permutahedra forming a quotient,
\item\label{it:gp} generalized $n$-polymatroids,
and
\item\label{it:dp} deformed $(n+1)$-permutahedra $\polR$ for which $\sum_{i=1}^{n+1} x_i = 0$ for all $\b x\in\polR$. %
\end{compactenum}

The bijection between (\ref{it:quotients}) and (\ref{it:gp}) is given by $(\pol, \polQ)\mapsto \GPsum$, its reciprocal by $\polR \mapsto \bigl(\polR^{+\b 1},\, \polR^{-\b 1}\bigr)$.
The bijection between (\ref{it:quotients}) and (\ref{it:dp}) is given by the map $(\pol, \polQ) \mapsto \, \lift{\!(\GPsum)}$, and its reciprocal by
$\polR \mapsto \bigr(\proj{(\polR^{-\b 1})},\, \proj{(\polR^{+\b 1})}\,\bigl)$.
\end{theorem}

\begin{proof}
\Cref{lem:GP-sumpolymatroid} shows that the map $(\pol, \polQ)\mapsto \GPsum$ is well defined, and \cref{lem:GeneralizedPolymatroisUniquelyDetermined} that it is a bijection. The second part is the result of combining this with \Cref{prop:LiftedGPsumProperties}.
\end{proof}

The bijections discussed in this section can be summarized in the following diagram. Recall that, if $\pol$ is a deformed $(n+1)$-permutahedron, then $\sum_{i=1}^{n+1}x_i$ is independent from $\b x\in \pol$. We can identify the set $\SC[n+1] \,/\, \R\b e_{n+1}$ of deformed $(n+1)$-permutahedra up to translation by a multiple of $\b e_{n+1}$ with 
 the set of deformed $(n+1)$-permutahedra for which $\sum_{i=1}^{n+1} x_i = 0$.
Here, we denote $\GP$ the set of generalized $n$-polymatroids.
Note that the kernel of the leftmost projection is generated by the translations by $\b e_{n+1}$, which is included in the lineality space of $\SC[n+1]$.

$$\begin{array}{ccccccl}
\SC[n+1] &\twoheadrightarrow& \SC[n+1] \,/\, \R\b e_{n+1} &\xrightarrow{\sim}& \GP &\xrightarrow{\sim}& \{(\pol, \polQ) ~;~ \pol, \polQ\in\SC^2,\, \text{ with } \quotientQP\} \vspace{0.15cm}\\
\pol &\mapsto& \pol - \left(\sum_{i=1}^{n+1}x_i\right)\b e_{n+1} &&&& \vspace{0.15cm}\\
&& \polS &\mapsto& \proj{\polS} && \vspace{0.15cm}\\
&&&&\polR &\mapsto& (\polR^{+\b 1},\, \polR^{-\b 1}) \\
&&&& \GPsum &\mapsfrom& (\pol, \polQ) \vspace{0.15cm}\\
&& \lift{\polR} &\mapsfrom& \polR &&
\end{array}$$

\subsection{Admissibility cones}\label{ssec:AdmissCones}

\begin{definition}
The \defn{admissibility cone} of two deformed $n$-permutahedra $\pol, \polQ$ with $\genquotientQP$ is:
$$\mathdefn{\admisscone} \coloneqq \bigl\{\GPsum[\polQ'][\pol'] ~;~ \pol'\deformed\pol,~ \polQ'\deformed\polQ, \text{ satisfying } \quotientQP[\pol'][\polQ']\bigr\}$$
\end{definition}

Our first observation is that $\admisscone$ is a polyhedral cone, linearly isomorphic to the intersection of $\DefoCone\times\DefoCone[\polQ]$ with certain linear half-spaces.
We will see that $\admisscone$ is linearly isomorphic to the deformation cone $\DefoCone[\GPsum]$ (which coincides with $\DefoCone[\liftGPsum]$ up to their linearity).

\begin{proposition}\label{prop:AdmissConeIsPolyhedral}
	For $\pol, \polQ\in \SC$, the cone $\admisscone$ is a polyhedral cone which is linearly isomorphic to the cone $\bigl(\DefoCone[\polQ] \times \DefoCone\bigr) \cap \bigl\{ (\polQ, \pol) ~;~ \quotientQP \bigr\}$.
\end{proposition}

The proof requires the next lemma, which is a reformulation of \cite[Theorem 14.2.15]{Frank:2011} using the characterization of quotients, as we now detail.
We give a self-contained (similar) argument.

\begin{figure}
	\centering
	\includegraphics[width=0.9\linewidth]{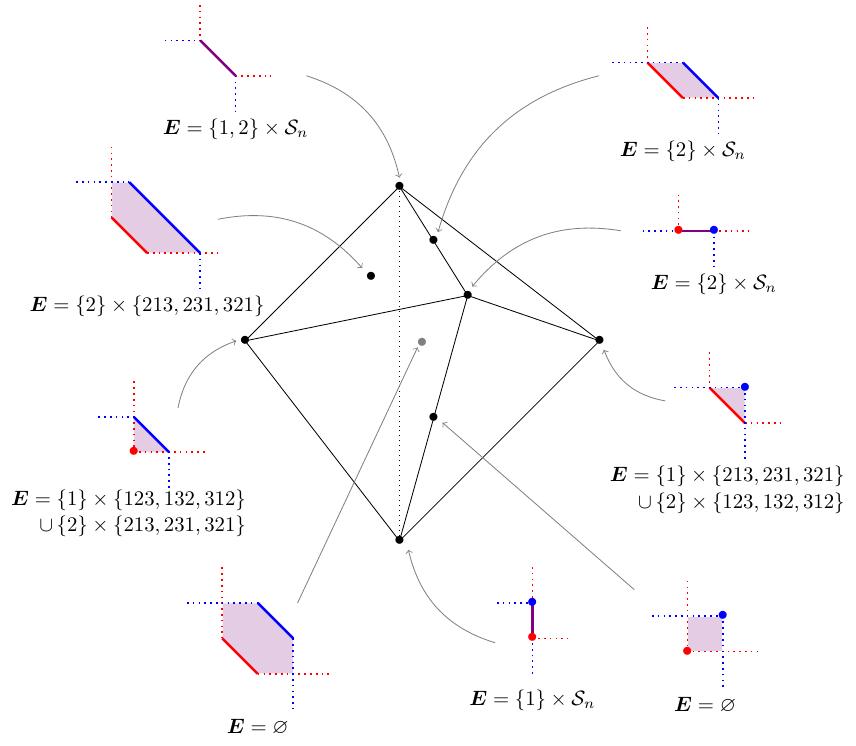}
	\caption[{The submodular cone $\SC[3]$, and various GP-sums}]{
    The notation $\b E(\pol, \polQ)$ will appear relevant in \Cref{sec:equalitysets}.
    The submodular cone $\SC[3]$ has dimension $7$.
    As its lineality space has dimension $3$, it is a cone over a polytope of dimension $3\, (=7-3-1)$. We depict this polytope (a bi-pyramid over a vertical triangle, shown at the center of the figure).
    Each point of this polytope corresponds to the GP-sum of two deformed $1$-permutahedra (which are segments or points).
    Each face corresponds to a class of normally equivalent GP-sums.
    We illustrate some of these GP-sums.}
	\label{fig:SC3}
\end{figure}

\begin{lemma}\label{lem:GPsumIsLinear}
The map $(\pol, \polQ) \mapsto \GPsum$ is linear (on virtual polytopes):
for deformed permutahedra $\pol_1, \pol_2$, $\polQ_1, \polQ_2$ and $\lambda,\mu\geq 0$, if $\quotientQP[\pol_1][\polQ_1]$ and $\quotientQP[\pol_2][\polQ_2]$, then $\quotientQP[(\lambda\pol_1 + \mu\pol_2)][(\lambda\polQ_1 + \mu\polQ_2)]$, and:
$$\GPsum[(\lambda\pol_1 + \mu\pol_2)][(\lambda\polQ_1 + \mu\polQ_2)] = \lambda(\GPsum[\pol_1][\polQ_1]) + \mu(\GPsum[\pol_2][\polQ_2])$$
\end{lemma}

\begin{proof}
Optimization and Minkowski sums commute: $(\pol + \polQ)^{\b c} = \pol^{\b c} + \polQ^{\b c}$.
Hence:
$$\Bigl(\lambda(\GPsum[\pol_1][\polQ_1]) + \mu(\GPsum[\pol_2][\polQ_2])\Bigr)^{+\b 1} = \Bigl(\lambda(\GPsum[\pol_1][\polQ_1])\Bigr)^{+\b 1} + \Bigl(\mu(\GPsum[\pol_2][\polQ_2])\Bigr)^{+\b 1} = \lambda\pol_1 + \mu\pol_2$$

The symmetric case holds for $-\b 1$ and $\lambda\polQ_1 + \mu\polQ_2$.

Besides, as $\quotientQP[\pol_1][\polQ_1]$ and $\quotientQP[\pol_2][\polQ_2]$, for all $i\in [n]$ and $\sigma\in \c S_n$, we have: $(\lambda\polQ_1+\mu\polQ_2)^\sigma_i = \lambda(\polQ_1)^\sigma_i + \mu(\polQ_2)^\sigma_i \leq \lambda(\pol_1)^\sigma_i + \mu(\pol_2)^\sigma_i = (\lambda\pol_1+\mu\pol_2)^\sigma_i$, hence we get that $\quotientQP[(\lambda\pol_1 + \mu\pol_2)][(\lambda\polQ_1 + \mu\polQ_2)]$.

Moreover, as $\lambda(\GPsum[\pol_1][\polQ_1]) + \mu(\GPsum[\pol_2][\polQ_2])$ is a sum of generalized $n$-polymatroids, it is itself a generalized $n$-polymatroid.
According to \Cref{lem:GeneralizedPolymatroisUniquelyDetermined}, generalized $n$-polymatroids are recognizable by their top and bottom faces, which implies the claimed equality.
\end{proof}

\begin{remark}
	Note that the map $(\pol, \polQ) \mapsto \GPsum$ \textbf{is not} bi-linear in general.
	For instance, $\GPsum[(\lambda\polQ)] \ne \lambda(\GPsum)$, and $(\GPsum[\polQ_1]) + (\GPsum[\polQ_2]) \ne \GPsum[(\polQ_1 + \polQ_2)]$.
	The correct formulae are: $\GPsum[(\lambda\polQ)][(\lambda\pol)] = \lambda(\GPsum)$, and $(\GPsum[\polQ_1]) + (\GPsum[\polQ_2]) = \GPsum[(\polQ_1 + \polQ_2)][(2\pol)]$.
\end{remark}

\begin{proof}[Proof of \cref{prop:AdmissConeIsPolyhedral}]
The map $\pol\mapsto \pol_i^\sigma$ is linear. Hence, the cone $\bigl\{ (\polQ',\pol')\in \SC^2 ~;~ \quotientQP[\pol'][\polQ']\bigr\}$ is polyhedral, because it is the intersection of the polyhedral cone $\SC^2$ with half-spaces of the form $(\pol')_i^\sigma\leq(\polQ')_i^\sigma$.
The image by the map $(\polQ,\pol) \mapsto \GPsum$ of the polyhedral cone $\bigl(\DefoCone[\polQ] \times \DefoCone\bigr) \cap \bigl\{ (\polQ',\pol')\in \SC ~;~ \quotientQP[\pol'][\polQ']\bigr\}$ is $\admisscone$, and this map is linear by \Cref{lem:GPsumIsLinear}.
\end{proof}

We can finally state the main result of this section.

\begin{theorem}\label{thm:AdmissibilityConeIsDeformationCone}
	For $\pol, \polQ\in \SC$ with $\genquotientQP$, we have:
	\begin{align*}
		\DefoCone[\GPsum]  &= \admisscone&\text{ and }&&
		\DefoCone[\liftGPsum]  &= \admisscone\oplus \R\b e_{n+1}.
	\end{align*}

\end{theorem}

\begin{proof}
The second identity follows from the first together with \Cref{thm:bijections}.

For the first identity, we will first prove that $\admisscone=\DefoCone[\polS]$ for a certain generalized polymatroid~$\polS$, and then we will prove that we can take $\GPsum$ in the place of $\polS$.

	There are finitely many normal equivalence classes of generalized polymatroids. Therefore, there are finitely many normal equivalence classes of polytopes within~$\admisscone$. Put $\polR_0=\GPsum$ and choose representatives $\polR_1,\dots,\polR_s$ of all the other normal equivalence classes of polytopes in~$\admisscone$. We define $\polS=\sum_{i=0}^s \polR_i$ as their Minkowski sum.
	
Let $\polR\in \admisscone$. By construction, it is normally equivalent to one of the $\polR_i$, which is a deformation of $\polS$. Since being a deformation only depends on the normal equivalence class (see \cref{def:deformation}), the polytope $\polR$ must be itself a deformation of $\polS$.
Thus $\admisscone\subseteq \DefoCone[\polS]$.
	
	For the converse inclusion, we note first that $\polS$ is a generalized polymatroid of the form $\polS=\GPsum[\polQ'][\pol']$ for some $\polQ'$ normally equivalent to $\polQ$ and some $\pol'$ normally equivalent to $\pol$ such that $\quotientQP[\pol][\polQ']$. Indeed, by definition of the admissibility cone, each of the summands defining $\polS$ is of the form $\polR_i=\GPsum[\polQ_i][\pol_i]$ for some  $\pol_i\deformed\pol$ and $\polQ_i\deformed\polQ$ satisfying $\quotientQP[\pol_i][\polQ_i]$. Hence, we have $\pol'=\polS^{+\b 1}=\sum_{i=0}^s\polR_i^{+\b 1}=\sum_{i=0}^s\pol_i$, which must be normally equivalent to $\pol$ because $\pol_0=\pol$ and all the other $\pol_i$ are deformations of~$\pol$. By the same argument $\polQ'$ is normally equivalent to $\polQ$. And $\quotientQP [\pol'][\polQ']$ because $\quotientQP[\pol_i][\polQ_i ]$ and this property is preserved by Minkowski sums.
	
	Now, for any deformation $\polR$ of $\polS$, we will have that $\polR^{+\b 1}$ is a deformation of $\pol$, because $\polR^{+\b 1}$ must be a deformation of $\pol'$, and $\pol$ and $\pol'$ are normally equivalent. Similarly, we have $\polR^{-\b 1}$ is a deformation of $\polQ$. Finally, note that since $\polS$ is a generalized polymatroid, so must be $\polR$, because the family of generalized $n$-polymatroids is closed under deformations. By \cref{lem:polymatroidtopbotquotient}, this implies that $\quotientQP[\polR^{+\b 1}][\polR^{-\b 1}]$. All together this shows that $\polR$ is in the admissibility cone. And hence we have proven that  $\admisscone\supseteq \DefoCone[\polS]$.

	Finally, we need to prove that $\polS$ and $\GPsum$ are normally equivalent (and hence have the same deformation cone). For this, note that $\GPsum$ lies in the interior of $\admisscone$ because it does not belong to any of its facet-defining hyperplanes ($\polQ$ is in the interior of $\DefoCone[\polQ]$, $\pol$ is in the interior of $\DefoCone$, and $\polQ^\sigma_i < \pol^\sigma_i$ for all $i\in [n]$ and $\sigma \in \c S_n$ because $\genquotientQP$). Since we already proved that $\admisscone=\DefoCone[\polS]$, this shows that $\GPsum$ belongs to the interior of $\DefoCone[\polS]$, and thus that it is normally equivalent to $\polS$. This concludes the proof.
\end{proof}

\begin{remark}
The hypothesis $\genquotientQP$ is fundamental, and cannot be replaced by the weaker $\quotientQP$.
For instance, take $\pol$ to be the segment $\bigl[(0, 1),\, (1, 0)\bigr]$ as in \Cref{fig:SC3} (top left), and $\polQ = \pol + (1, 1)$.
We have $\genquotientQP$ and $\DefoCone[\GPsum] = \SC[3]$.
On the other side, we have $\quotientQP[\pol][\pol]$, and $\DefoCone[{\GPsum[\pol][\pol]}] \subsetneq \SC[3]$ (as $\GPsum[\pol][\pol] = \pol$ is a segment, the cone $\DefoCone[{\GPsum[\pol][\pol]}]$ is actually a ray). In \Cref{sec:equalitysets}, we expand the concept of admissibility cones to non-generic quotients.
\end{remark}

\section{Many rays of the submodular cone}\label{sec:Rays}

The goal of the present section is to establish a new bound on the number of rays of the submodular cone $\SC$.
Recall that a polytope is ``a ray of $\SC$'' if it is a Minkowski indecomposable deformed $n$-permutahedron, that is to say a polytope $\pol$ whose edges are all in directions $\b e_i - \b e_j$ for some $i\ne j$, and that cannot be written as $\pol = \polQ + \polR$ except if $\polQ$ and $\polR$ are translated dilates of $\pol$.
Equivalently, $\pol\in\SC$ is a ray of $\SC$ if and only if $\dim\DefoCone = n+1$.

Using the ideas underlying \Cref{sec:AdmissCone}, we present in this section an inductive construction of rays of the submodular cone: from two rays $\pol, \polQ$ of $\SC$, we will construct one ray of $\SC[n+1]$ as $\lift{\bigl(\GPsum[(\mu\polQ + \b t)]}\bigr)$ for well-chosen $\mu > 0$ and $\b t\in \R^n$.

Before addressing the general case, we look at the case where $\polQ$ is a point.

\subsection{Independence polytopes}\label{sec:indpol}

We first consider the case where $\polQ$ is a point, as it will reveal to be very instructive. After a suitable translation if needed, we can assume $\polQ = \b 0$.
Note that setting $\pol = \b 0$ instead only amounts to applying a central symmetry.

\begin{definition}
For a deformed $n$-permutahedron $\pol\in\SC$, let $\b t\in \R^n$ be defined by $t_i = \min_{\b x\in \pol}x_i$.
The \defn{independence polytope} of $\pol$ is $\mathdefn{\pol[Ind]_{\pol}} \coloneqq \GPsum[\b 0][(\pol - \b t)] = \R_+^n\cap (\pol - \b t +\R_-^n)$.
\end{definition}

The name is motivated from terminology for matroids, while without the translation vector $\b t$, they are known by the name \emph{polymatroid polytope}. 
More precisely, if $\c M$ is a matroid without coloops, and $\pol_{\c M} = \conv\{\b e_B ~;~ B \subseteq[n] \text{ is a basis in } \c M\}$ denotes the associated \defn{matroid (base) polytope}, then $\pol[Ind]_{\pol_{\c M}}$ is the \defn{independence matroid polytope} associated to $\c M$, that is $\pol[Ind]_{\pol_{\c M}} = \conv\{\b e_S ~;~ S \subseteq[n] \text{ is an independent set in } \c M\}$.

\begin{lemma}\label{lem:IndependencePolytopeSum}
For any deformed $n$-permutahedra $\pol$ and $\polQ$, we have $\pol[Ind]_{\pol + \polQ} = \pol[Ind]_{\pol} + \pol[Ind]_{\polQ}$.
\end{lemma}

\begin{proof}
Pick $\b t, \b t'\in\R^n$ defined by $t_i = \min_{\b x\in \pol}x_i$ and $t'_i = \min_{\b x\in \polQ}x_i$.
Then, immediately, $\b u = \b t + \b t'$ satisfies $u_i = \min_{\b x\in \pol+\polQ}x_i$.
By \Cref{lem:GPsumIsLinear}, we have that: $\pol[Ind]_{\pol} + \pol[Ind]_{\polQ} =  \bigl(\GPsum[(\pol-\b t)][\b 0]\bigr) +  \bigl(\GPsum[(\polQ-\b t')][\b 0]\bigr) = \GPsum[(\pol+\polQ - \b u)][(\b 0 + \b 0)] = \pol[Ind]_{\pol + \polQ}$.
\end{proof}

\begin{theorem}\label{thm:IndependencePolytopesEasyVersion}
Let $\pol$ be a deformed $n$-permutahedron. The following are equivalent:
\begin{compactenum}
    \item $\pol$ is indecomposable,
    \item the generalized $n$-polymatroid  $\pol[Ind]_{\pol}$ is indecomposable,
    \item  the deformed $(n+1)$-permutahedron $\,\lift{\!(\pol[Ind]_{\pol})}$ is indecomposable.
\end{compactenum}
\end{theorem}

\begin{proof}
By definition, if $\pol$ is a deformed $n$-permutahedron, then $\pol[Ind]_{\pol}$ is a generalized $n$-polymatroid, and $\,\lift{\!(\pol[Ind]_{\pol})}$ is a deformed $(n+1)$-permutahedron.
We need to prove that (in)decomposability carries over: by the construction of lifts, $\,\lift{\!(\pol[Ind]_{\pol})}$ is indecomposable if and only if $\pol[Ind]_{\pol}$ is indecomposable, since indecomposability is invariant under linear isomorphism.

If $\pol = \polQ + \polR$ where $\polQ$ is not a translated dilate of $\pol$, then by \Cref{lem:IndependencePolytopeSum}, we have the equality $\pol[Ind]_{\pol} = \pol[Ind]_{\polQ} + \pol[Ind]_{\polR}$ where $\pol[Ind]_{\polQ}$ is not a translated dilate of $\pol[Ind]_{\pol}$ (because the first has a face $\polQ$ while the second has a face $\pol$).
Thus, if $\pol$ is decomposable, then $\pol[Ind]_{\pol}$ is decomposable too.

Now, let $\b t\in \R^n$ with $t_i = \min_{\b x\in \pol} x_i$.
All facets of $\pol[Ind]_{\pol}$ contain a vertex of $\pol-\b t$: indeed, a facet of $\pol[Ind]_{\pol}$ is either $\pol-\b t$ itself, or a facet of $(\pol-\b t)+\R_-^n$ (intersected with $\b 0 + \R^n_+$), or a facet of $\b 0+\R_+^n$ (intersected with $(\pol-\b t) + \R^n_-$).
Those of the two first kinds clearly contain a vertex of $\pol-\b t$, while the latter are supported by some hyperplane $\{\b x\in \R^n ~;~ x_i = 0\}$ for some $i\in [n]$: this facet contains the vertex $\b p - \b t$ of $\pol - \b t$ for any $\b p$ with $p_i = \min_{\b x\in\pol}x_i$, by definition of $\b t$.

By McMullen's criterion \cite[Theorem~2]{McMullen1987} (see also \cite{PadrolPoullot-NewCriteriaIndecomposability} for more details), if a polytope admits an indecomposable face which shares a vertex with all its facets, then this polytope is indecomposable.
Consequently, if $\pol$ is indecomposable, then $\pol[Ind]_{\pol}$ is indecomposable.
\end{proof}

\begin{remark}
In the case of a matroid $\c M$, the polytope $\pol_{\c M}$ is indecomposable if and only if up to loops and coloops $\c M$ is connected (\ie cannot be written as a direct sum of matroids), see \cite[Section 10]{Nguyen1986-SemimodularFunctions} or \cite[Section 3.1]{PadrolPoullot-NewCriteriaIndecomposability}.
By \Cref{thm:IndependencePolytopesEasyVersion}, the independence polytope of $\c M$ is indecomposable if and only if $\c M$ connected.
\end{remark}

\Cref{thm:IndependencePolytopesEasyVersion} shows that the map $\pol\mapsto \,\lift{\!(\pol[Ind]_{\pol})}$ sends injectively any ray of $\SC$ to a ray of $\SC[n+1]$ (injectivity holds since $\pol = \pol[Ind]_{\pol}^{+\b 1}$).
Thus, the number of rays of $\SC$ grows with $n$.
As we want to prove that this growth is at least doubly exponential, we will need to prove that (almost) any \textbf{pair} of rays of $\SC$ can be mapped injectively to a ray of $\SC[n+1]$.
This is the aim of the next section.

\subsection{Inductive construction of rays}\label{ssec:IndependencePolytopes}
In this section, we generalize the independence polytopes to construct new indecomposable deformed permutahedra. To this end, throughout the section we suppose that $\pol$ and $\polQ$ are indecomposable deformed permutahedra which are not points.

In this case, their admissibility cone $\admisscone$ has a particularly simple structure. Since all deformations of an indecomposable polytope are dilated translates, one obtains:
\begin{align}
\admisscone &= \left\{\GPsum[\left(\mu \polQ + \b v\right)][\left(\lambda \pol + \b u\right)] ~;~ \begin{array}{l}
\lambda, \mu \geq 0 \\
\b u, \b v \in \R^n
\end{array},\,\forall \sigma \in \c S_n \  \forall i \in [n]: \left(\mu \polQ + \b v\right)^\sigma_i \leq \left(\lambda \pol + \b u\right)^\sigma_i \right\}\nonumber\\
& \simeq  \bigl\{(\lambda, \mu, \b u, \b v) \,;\, \lambda, \mu \geq 0,\, \b u, \b v \in \R^n,\, \forall \sigma \in \c S_n \ \forall i \in [n]: \mu\polQ^\sigma_i + v_i \leq \lambda\pol^\sigma_i + u_i \bigr\} 
\label{eq:admisscone}
\end{align}
where $\simeq$ denotes linear isomorphism. For this statement, we used the following lemma: 

\begin{lemma}
The map $(\lambda, \mu, \b u, \b v) \mapsto \GPsum[(\mu \polQ + \b v)][(\lambda \pol + \b u)]$ is linear.
\end{lemma}

\begin{proof}
The maps $(\lambda,\mu, \b u,\b v) \mapsto (\lambda \pol+\b u,\, \mu\polQ+\b v)$,
and $(\pol, \polQ) \mapsto \GPsum$ are linear (\Cref{lem:GPsumIsLinear}).
\end{proof}

We can simplify the description of the admissibility cone further, emphasizing its lineality space.
For fixed $\pol,\polQ$, we define the map $\mathdefn{\b \phi }\colon \R^2 \to \R^n$ with \defn{$\b \phi(\lambda, \mu)_i$} $ \coloneqq \min_{\sigma \in \c S_n}\left(\lambda \pol^\sigma_i - \mu \polQ^\sigma_i\right)$, and set:
\begin{align*}
\mathdefn{\Iadmisscone} &\coloneqq
\left\{(\lambda, \mu, \b t) \,;\, \lambda, \mu \geq 0, \text{ and }\forall \sigma \in \c S_n \  \forall i \in [n],\quad t_i \leq \lambda \pol^\sigma_i - \mu \polQ^\sigma_i \right\}\\
&=
\left\{(\lambda, \mu, \b t) \,;\, \lambda, \mu \geq 0, \text{ and }\forall i \in [n],\quad t_i \leq \b \phi(\lambda, \mu)_i \right\}.
\end{align*}

Then, setting $\b t = \b v - \b u$ in \eqref{eq:admisscone}, we get:
\begin{align*}
\admisscone & \simeq  \Iadmisscone  \,\oplus\, \R^n.
\end{align*}

Note that $\dim\admisscone = 2n + 2 = \dim\DefoCone + \dim\DefoCone[\polQ]$. %
Moreover, $\admisscone$ has a lineality space of dimension $n$, as $(\GPsum) \subseteq \R^n$.
Hence $\Iadmisscone$ is a pointed polyhedral cone.
We are interested in the rays of $\Iadmisscone$ for which $\lambda \ne 0$ and $\mu \ne 0$ (\ie both are strictly positive).

\begin{definition}\label{def:Seed}
Let $\pol$ and $\polQ$ be indecomposable deformed $n$-permutahedra. A couple $(\lambda, \mu)$ with $\lambda, \mu > 0$ is a \defn{seed} for $(\pol,\polQ)$ if there exists an index $j \in [n]$ and\footnote{The notation ``$\argmin_{x\in X} f(x)$'' designates the \underline{set} of minimizers of $f$ in $X$.} two different permutations $\tau, \tau' \in \argmin_{\sigma \in \c S_n}\left(\lambda \pol^\sigma_j - \mu \polQ^\sigma_j\right)$ satisfying $(\pol^\tau_j, \polQ^\tau_j) \ne (\pol^{\tau'}_j, \polQ^{\tau'}_j)$.
\end{definition}

\begin{remark}\label{rmk:SeedOperations}
For two deformed permutahedra $\pol, \polQ$, if $(\lambda,\mu)$ is a seed for the pair of indecomposable deformed $n$-permutahedra $(\pol,\polQ)$, and  $\pol' = \alpha\pol+\b t$ and $\polQ' = \beta\polQ + \b u$, for some $\alpha, \beta > 0$ and $\b t,\b u\in \R^n$, then $(\frac{\lambda}{\alpha}, \frac{\mu}{\beta})$ is a seed for $(\pol', \polQ')$.
\end{remark}

We will first showcase in \Cref{prop:SeedIsRay} (and \Cref{exm:FertilitySmallExample}) what seeds are useful for, then give an easy-to-check criterion on the pair $(\pol, \polQ)$ for the (in)existence of seeds in \Cref{thm:NotFertile}.

\begin{proposition}\label{prop:SeedIsRay}
Let $\pol$ and $\polQ$ be indecomposable generalized permutahedra and $\lambda, \mu > 0$.
If $(\lambda, \mu)$ is a seed, then 
$\GPsum[\bigl(\mu\polQ + \b\phi(\lambda,\mu)\bigr)][\bigl(\lambda\pol\bigr)]$ is a ray of $\admisscone$ (whose top and bottom faces are normally equivalent to $\pol$ and~$\polQ$, respectively); that is, $\GPsum[\bigl(\mu\polQ + \b\phi(\lambda,\mu)\bigr)][\bigl(\lambda\pol\bigr)]$  is indecomposable.
\end{proposition}

\begin{proof}
This is equivalent to showing that $\bigl(\lambda, \mu, \b\phi(\lambda, \mu)\bigr)$ is a ray of $\Iadmisscone$.
Firstly, by definition of $\b\phi$, note that $\bigl(\lambda, \mu, \b\phi(\lambda, \mu)\bigr)\in\Iadmisscone$.
Moreover, if $\lambda\ne 0$, then $\bigl(\lambda, \mu, \b\phi(\lambda, \mu)\bigr) \ne \b 0$.

As $(\lambda, \mu)$ is a seed, let $j\in [n]$ be such that
$\tau, \tau' \in \argmin_{\sigma \in \c S_n}\left(\lambda \pol^\sigma_j - \mu \polQ^\sigma_j\right)$ with $\tau\neq \tau'$ (which implies $\lambda(\pol^{\tau'}_j - \pol^{\tau}_j) = \mu(\polQ^{\tau'}_j - \polQ^{\tau}_j)$), with $(\pol^\tau_j, \polQ^\tau_j) \ne (\pol^{\tau'}_j, \polQ^{\tau'}_j)$. 
Moreover, as $\lambda\ne0$ and $\mu\ne0$, having $(\pol^\tau_j, \polQ^\tau_j) \ne (\pol^{\tau'}_j, \polQ^{\tau'}_j)$ implies both $\pol^\tau_j \ne \pol^{\tau'}_j$ and $\polQ^\tau_j \ne \polQ^{\tau'}_j$.
We get that $\mu = \lambda \frac{\pol^{\tau'}_j - \pol^{\tau}_j}{\polQ^{\tau'}_j - \polQ^{\tau}_j}$.
Thus, $\b r_\lambda \coloneqq\bigl(\lambda, \mu, \b\phi(\lambda, \mu)\bigr)$ is fully determined by $\lambda$.
We now prove that $\b r_\lambda$ satisfies $n+1$ independent equalities among the ones defining $\Iadmisscone$.
First, recall that $\Iadmisscone =
\left\{(\lambda, \mu, \b t) \,;\, \lambda, \mu \geq 0,\, \forall \sigma \in \c S_n, \forall i\in [n],\, t_i \leq \lambda\pol^\sigma_i - \mu\polQ^\sigma_i\right\}$.
For any $\lambda > 0$ and above-defined $\mu$, for each $i\ne j$, let $\sigma_i\in S_n$ be such that $\lambda\pol_i^{\sigma_i}-\mu\polQ_i^{\sigma_i} = \min_{\sigma \in \c S_n}\left(\lambda \pol^\sigma_i - \mu \polQ^\sigma_i\right)$.
The vector $\b r_\lambda$ satisfies the $n-1$ inequalities of the form ``$t_i \leq \lambda\pol_i^{\sigma_i}-\mu\polQ_i^{\sigma_i}$'' with equality, as well as the $2$ inequalities ``$t_j \leq \lambda\pol_j^{\tau}-\mu\polQ_j^{\tau}$'' and ``$t_j \leq \lambda\pol_j^{\tau'}-\mu\polQ_j^{\tau'}$'' with equality.
As $(\pol^\tau_j, \polQ^\tau_j) \ne (\pol^{\tau'}_j, \polQ^{\tau'}_j)$, these $n+1$ equations in the $n+2$ variables $\b t, \lambda, \mu$ are linearly independent.
Hence $\b r_\lambda$ is a ray of $\Iadmisscone$.

Note that, by \Cref{thm:AdmissibilityConeIsDeformationCone,eq:admisscone}, this implies that $\GPsum[\bigl(\mu\polQ + \b\phi(\lambda,\mu)\bigr)][\bigl(\lambda\pol\bigr)]$ is a ray of a deformation cone, and hence indecomposable.
Finally, as $\lambda, \mu> 0$, the top and bottom faces of $\GPsum[\bigl(\mu\polQ + \b\phi(\lambda,\mu)\bigr)][\bigl(\lambda\pol\bigr)]$ are normally equivalent to $\pol$ and $\polQ$ respectively.
\end{proof}

Our aim is to associate a ray of the $(n+1)$-submodular cone $\SC[n+1]$ to a pair of rays of the $n$-submodular cone $\SC$.
That is precisely the role of seeds:
If a pair of indecomposable deformed $n$-permutahedra $(\pol, \polQ)´$ admits a seed $(\lambda, \mu)$, then we can create indecomposable deformed $(n+1)$-permutahedra by lifting the generalized $n$-polymatroid $\GPsum[\bigl(\mu\polQ + \b\phi(\lambda,\mu)\bigr)][\bigl(\lambda\pol\bigr)]$ to $\R^{n+1}$.

However, one technical problem remains:
Do seeds exist?
\Cref{exm:FertilitySmallExample} shows that it is not always the case; but we will prove that they do exist for enough pairs.

\begin{definition}\label{def:FertilePair}
For two indecomposable deformed $n$-permutahedra $\pol, \polQ$, their \defn{seed-set} is defined as $\Lambda(\pol, \polQ) \coloneqq \bigl\{(\lambda, \mu) \,;\, \lambda, \mu > 0 \,\text{and}\, (\lambda, \mu) \,\text{is a seed of}\, (\pol, \polQ)\bigr\}$.
The pair $(\pol, \polQ)$ is \defn{fertile} if $\Lambda(\pol, \polQ) \ne \emptyset$.%

Note that $(\pol, \pol)$ is always fertile, as $(1, 1)$ is a seed, with $\b\phi(1, 1) = \b 0$.
Note also that the fact that $(\pol, \polQ)$ is fertile \textbf{does not} imply that $(\polQ, \pol)$ is fertile, see \Cref{exm:FertilitySmallExample}.

\end{definition}

\begin{remark}\label{rmk:fertile_translationdilation}
By \Cref{rmk:SeedOperations}, fertility is preserved under translations and dilations.
\end{remark}

\begin{example}\label{exm:FertilitySmallExample}
Up to dilation and translation, there are 5 indecomposable deformed $3$-permutahedra, namely 3 segments and 2 triangles:
\begin{align*}
\polS_1 &= \conv\bigl(\b 0,\, (1, -1, 0)\bigr);\qquad  \polS_2 = \conv\bigl(\b 0,\, (0, 1, -1)\bigr);\qquad  \polS_3 = \conv\bigl(\b 0,\, (1, 0, -1)\bigr)\\
\polT_1 &= \conv\bigl((1, 0, 0),\, (0, 1, 0),\, (0, 0, 1)\bigr) ;\qquad  \polT_2 = \conv\bigl((-1, 0, 0),\, (0, -1, 0),\, (0, 0, -1)\bigr)
\end{align*}
See \Cref{fig:S1T1braidFan} for an illustration of $\polT_1$ and $\polS_1$.

\begin{figure}
	\centering
	\includegraphics[scale=1]{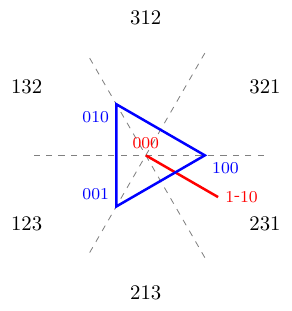}
	\caption[Triangle $\polT_1$ and segment $\polS_1$ in the braid fan]{The triangle $\polT_1$ in \textcolor{blue}{blue}, and the segment $\polS_1$ in \textcolor{red}{red}. Both lie in $\R^3$, but as they belong to the same plane, we picture it in 2-dimensions.
		In \textcolor{blue}{blue} and \textcolor{red}{red} are the coordinates of the vertices. In black is depicted the braid fan: the dashed lines are the hyperplanes $\{\b x\in \R^3 ~;~ x_i = x_j\}$ for $i\ne j$ (intersected with our plane of embedding), and the each maximal region is a cone $\polC_\sigma$ labeled by its corresponding permutation $\sigma\in \c S_3$.}
	\label{fig:S1T1braidFan}
\end{figure}

\begin{figure}[b]
	\begin{center}
		\begin{tikzpicture}[decoration={markings, mark=at position 0.5 with {\draw[very thick, red] (-3pt,-3pt) -- (0,0);\draw[very thick, red] (-3pt,3pt) -- (0,0);}}]
			
			\coordinate (t1) at (0, 0);
			\coordinate (t2) at (3, 0);
			\coordinate (s1) at (1.5, 1);
			\coordinate (s2) at (1.5, 2);
			\coordinate (s3) at (1.5, 3);
			
			\draw[postaction={decorate}] (t2) -- (t1);
			\draw[postaction={decorate}] (t2) -- (s1);
			\draw[postaction={decorate}] (s1) -- (t1);
			\draw[postaction={decorate}] (t2) -- (s2);
			\draw[postaction={decorate}] (s2) -- (t1);
			\draw[postaction={decorate}] (t2) -- (s3);
			\draw[postaction={decorate}] (s3) -- (t1);
			
			\draw (t1) node{$\bullet$};
			\draw (t2) node{$\bullet$};
			\draw (s1) node{$\bullet$};
			\draw (s2) node{$\bullet$};
			\draw (s3) node{$\bullet$};
			
			\draw (t1) node[left]{$\polT_1$};
			\draw (t2) node[right]{$\polT_2$};
			\draw (s1) node[above]{$\polS_1$};
			\draw (s2) node[above]{$\polS_2$};
			\draw (s3) node[above]{$\polS_3$};
		\end{tikzpicture}
	\end{center}
	\caption{Fertile pairs among indecomposable deformed $3$-permutahedra}\label{fig:3fertility}
\end{figure}
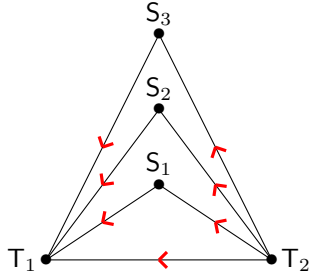

\Cref{fig:3fertility} displays a directed graph in which there is a directed edge $\pol-\!\!\!-\!\!\!\!\!\textcolor{red}{\textbf{>}}\!\!\!-\!\!\!\!-\polQ$ if and only if the pair $(\pol, \polQ)$ is fertile.
	For any~$\pol$, the pair $(\pol, \pol)$ is fertile, so our graph lacks loops on its vertices which we did not display for sake of readability.
	In this example, there is no $\pol\neq  \polQ$ such that both $(\pol, \polQ)$ and $(\polQ, \pol)$ are fertile: \Cref{exm:FamillyOfPairwiseFertileN4} will present the opposite phenomenon.
	
	 For each one of these fertile pairs $(\pol, \polQ)$, we can compute $\Lambda(\pol, \polQ)$.
	Firstly, for each $i\in [3]$, we compute the $6$ functions $f_{i,\sigma}: \lambda\mapsto \lambda\pol^\sigma_i - \polQ^\sigma_i$.
	Recall that the $i^{\text{th}}$-coordinate of $\b \phi(\lambda, 1)$ is the minimum $\min_{\sigma\in \c S_3}f_{i, \sigma}(\lambda)$: hence, a pair $(\lambda_\circ, 1)$ is a seed if and only if $\lambda_\circ$ is the abscissa of an intersection point of these affine functions $f_{i, \sigma}$ whose cordinate achieves the minimum of $\{f_{i,\sigma}(\lambda_\circ) ~;~ \sigma\in \c S_3\}$.
	For example, these intersection points are easy to spot on \Cref{fig:Hypograph} (they all have abscissa 1).
	A pair $(\lambda, \mu)$ is a seed of $(\pol, \polQ)$ if and only if $(\frac{\lambda}{\mu}, 1)$ is a seed.

For all the fertile pairs $(\pol, \polQ)\in \SC[3]\times\SC[3]$, there is a unique seed up to dilation: $(\lambda, \mu) = (1, 1)$, \ie the seed-set is equal to $\Lambda(\pol, \polQ) = \{(\lambda, \lambda) \,;\, \lambda > 0\}$.
Hence, we can look at the indecomposable deformed $4$-permutahedra $\liftGPsum[\bigl(\polQ + \b\phi(1, 1)\bigr)][\pol]$, where:

\begin{center}
\begin{tabular}{lc|ccccc}
&& \multicolumn{5}{c}{$\polQ$} \\
& $\b\phi(1, 1)$ & $\polS_1$ & $\polS_2$ & $\polS_3$ & $\polT_1$ & $\polT_2$\\ \hline
&$\polS_1$ & $\b0$ & & & $(0, -1, -1)$ & \\ 
&$\polS_2$ & & $\b0$ & & $(0, -1, -1)$ & \\ 
$\pol$&$\polS_3$ & & & $\b0$ & $(-1, 0, -1)$ & \\
&$\polT_1$ & & & & $\b0$ & \\ 
&$\polT_2$ & $(-1, 0, -1)$ & $(-1, -1, 0)$ & $(-1, -1, 0)$ & $(-1, -1, -1)$ & $\b0$ \\ 
\end{tabular}%
\end{center}
\smallskip 

Firstly, we have directly, $\liftGPsum[\pol] = \pol \times\{0\}$, which are segments and triangles.
Besides:

\noindent$\liftGPsum[\bigr(\polT_1 + (0, -1, -1)\bigl)][\polS_1] = \conv\Bigl(
(0, -1, 0, 1),
(0, 0, -1, 1),
(0, 0, 0, 0),
(1, -1, -1, 1),
(1, -1, 0, 0)\Bigr)$ is a square pyramid.

\noindent$\liftGPsum[\bigr(\polT_1 + (-1, -1, -1)\bigl)][\polT_2] = \conv\Bigl(
(-1, -1, 0, 2),
(-1, 0, -1, 2),
(-1, 0, 0, 1),
(0, -1, -1, 2),$\linebreak$(0, -1, 0, 1), (0, 0, -1, 1)\Bigr)$ is an octahedron.

All other fertile pairs give square pyramids which are rotations of $\liftGPsum[\bigr(\polT_1 + (0, -1, -1)\bigl)][\polS_1]$.

\begin{figure}
\centering
\includegraphics[width=0.99\linewidth]{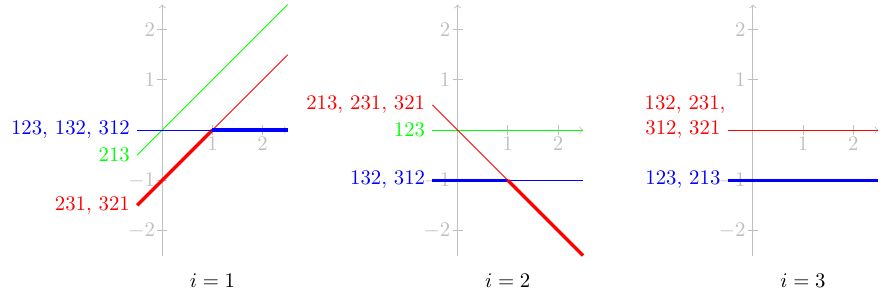}
\caption[Seeds correspond to abscissa of bottommost intersection points]{Computation of the seed set and the corresponding values of $\b \phi(\lambda, \mu)$ for the pair $(\polS_1, \polT_1)$: a seed is a pair $(\lambda_\circ, 1)$ where $\lambda_\circ$ is the abscissa of an intersection point with smallest corresponding ordinate.}
\label{fig:Hypograph}
\end{figure}

\end{example}

To prove that we have ``a lot'' of fertile pairs, we first prove that fertility is hereditary, and we will show in \Cref{exm:FamillyOfPairwiseFertileN4,exm:FamillyOfPairwiseFertileN5} that fertility is well-spread among pairs of deformed $n$-permutahedra for $n = 4$ and $n = 5$.

\begin{theorem}\label{thm:FertilityIsHeriditary}
Let $\polS, \pol$ be indecomposable generalized permutahedra, and $\polQ,\polR\in\SC$ be quotients $\polQ\isquotient\pol$ and $\polS\isquotient\polR$ such that the lifted GP-sums $\liftGPsum[\polS][\polR]$ and $\liftGPsum$ are indecomposable. If the pair $(\polS, \pol)$ is fertile, then the pair $\bigl(\liftGPsum[\polS][\polR],\, \liftGPsum\bigr)\in\SC[n+1]^2$ is also fertile.
\end{theorem}

Before proving this theorem, we need two short lemmas and a theorem.

\begin{lemma}\label{lem:MinForPandQ}
For two deformed $n$-permutahedra $\pol, \polQ$ and $i\in[n]$, there exists $\tau\in \c S_n$ such that $\pol^\tau_i = \max_{\sigma\in \c S_n}\pol^\sigma_i$ and $\polQ^\tau_i = \max_{\sigma\in \c S_n}\polQ^\sigma_i$.
\end{lemma}

\begin{proof}
We have $\pol^\tau_i = \max_{\sigma\in \c S_n}\pol^\sigma_i$ if and only if $\b e_i$ is in the (outer) normal cone of the vertex~$\pol^\tau$.
Hence, pick any $\tau$ with $\tau(1) = i$, then $\pol^\tau_i = \min_{\sigma\in \c S_n}\pol^\sigma_i$ for any deformed $n$-permutahedron $\pol$ (consequently, there exists $\tau$ maximizing for all deformed $n$-permutahedra at the same time).
\end{proof}

\begin{remark}%
This lemma actually holds in a broader generality:
If $\pol$ and $\polQ$ are two deformations of the same polytope $\pol[A]$, then for any $\b c$, there exists $\b a$ for which $\pol[A]^{\b a}$ is a vertex (implying that $\pol^{\b a}$ and $\polQ^{\b a}$ are also vertices), and such that $\inner{\pol^{\b a}, \b c} = \max_{\b x\in \pol}\inner{\b x, \b c}$ and $\inner{\polQ^{\b a}, \b c} = \max_{\b y\in \polQ}\inner{\b y, \b c}$.
Such $\b a$ can be chosen arbitrarily such that $\b c$ is in the normal cone of the vertex $\pol[A]^{\b a}$.

We will only use the case of deformed $n$-permutahedra in what follows.
\end{remark}

In the next proof, we identify a permutation $\tau\in \c S_{n+1}$ with a total order $\tau(1)<\cdots<\tau(n+1)$, and denote by $\tau^\downarrow\in\c S_ n$ the permutation induced by $\tau$ on $[n]$ by restricting the total order to $[n]$.

\begin{lemma}\label{lem:MaximizerInLowerDimension}
	Suppose $\polQ\isquotient\pol$. For $i\in [n]$ (in particular, $i\ne n+1$), if $\tau\in \c S_{n+1}$ satisfies $(\liftGPsum)^\tau_i = \max_{\sigma\in \c S_{n+1}} (\liftGPsum)^\sigma_i$, then $\pol^{\tau^\downarrow}_i = \max_{\sigma\in \c S_n} \pol^\sigma_i$. Symmetrically, if $\tau\in \c S_{n+1}$ satisfies $(\liftGPsum)^\tau_i = \min_{\sigma\in \c S_{n+1}} (\liftGPsum)^\sigma_i$, then $\polQ^{\tau^\downarrow}_i = \min_{\sigma\in \c S_n} \polQ^\sigma_i$.
\end{lemma}

\begin{proof}

Note that $\sum_{i\in [n]}x_i$ is constant for $\b x\in \polQ,\pol$, and that $\sum_{i\in [n]}x_i\leq \sum_{i\in [n]}y_i$ for any $\b x\in \polQ$ and $\b y\in \pol$ because $\polQ\isquotient\pol$ (since $\pol = (\GPsum)^{+\b 1}$).
Moreover, $\polQ\isquotient\pol$ also implies that $\polQ^\tau_i \leq \pol^\tau_i$ for all $\tau, i$.
Hence, $\max_{\sigma\in \c S_{n+1}} (\liftGPsum)^\sigma_i=\max_{\sigma\in \c S_{n+1}} (\lift{\pol})^\sigma_i$. 
Besides, since the $(n+1)^{\text{th}}$ coordinate is constant over $\lift{\!\pol}$, if $\tau\in \c S_{n+1}$ maximizes $(\liftGPsum)^\tau_i$, then $(\GPsum)^{\tau^\downarrow}_i$  maximizes $(\GPsum)^\tau_i$.
	
The proof for $\polQ$ is analogous.
\end{proof}

\begin{figure}
\centering
\includegraphics[width=0.99\linewidth]{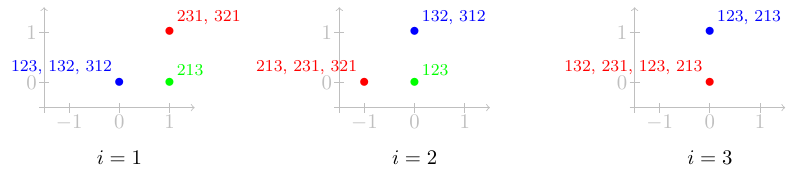}
\caption[Detecting fertility of a pair $(\pol, \polQ)$]{For the pair $(\pol, \polQ) = (\polS_1, \polT_1)$, for each $\sigma\in \c S_3$ and $i\in [3]$, the point $(\pol^\sigma_i, \polQ^\sigma_i)$.
By \Cref{lem:MinForPandQ}, for each $i\in [3]$, there is always a top-right-most and a bottom-left-most point in the drawing.
By \Cref{thm:NotFertile}, the pair $(\pol, \polQ)$ is fertile if and only if there exists $i\in [3]$ such that there is no point on the top-left corner.
Hence $(\pol, \polQ)$ and $(\polQ, \pol)$ are both fertile if and only if, for some $i\in [3]$, the drawing has only two corners points (top-right and bottom-left, but not top-left nor bottom-right).
Here $(\polS_1, \polT_1)$ is fertile but not $(\polT_1, \polS_1)$.}
\label{fig:MinMaxS1T1}
\end{figure}

We finally state our tool for characterizing the existence of seeds, both in practice and in theory (see Example~\ref{ex:fertile} below for an application of \Cref{thm:NotFertile}).

\begin{theorem}\label{thm:NotFertile}
The pair $(\pol, \polQ)$ is not fertile if and only if for all $i\in [n]$, there exists a permutation $\tau\in \c S_n$ such that $\pol^\tau_i = \min_{\sigma\in \c S_n} \pol^\sigma_i$ and $\polQ^\tau_i = \max_{\sigma\in \c S_n} \polQ^\sigma_i$.
\end{theorem}

\begin{proof}
To ease notation, we shorten ``$\min_{\sigma\in \c S_n}$'' to ``$\min$'', and idem for ``$\argmin$'', ``$\max$'' and ``$\argmax$''.
Recall that the map $\b \phi = \R^2 \to \R^n$ is defined by $\b \phi(\lambda, \mu)_i = \min\left(\lambda \pol^\sigma_i - \mu \polQ^\sigma_i\right)$. 
As~$\c S_n$ is finite, $f: \lambda \mapsto \b\phi(\lambda, 1)_i$ is a piecewise affine function (we define it for all $\lambda\in \R$, not only for $\lambda > 0$).
A value $\lambda_\circ\in \R$ is a break-point of $f$ (\ie an endpoint of an interval on which $f$ is affine) if and only if there exist $\tau, \tau'\in \argmin(\lambda_\circ\pol^{\sigma}_i - \polQ^{\sigma}_i)$ with $(\pol^\tau_i, \polQ^\tau_i) \ne (\pol^{\tau'}_i, \polQ^{\tau'}_i)$.
Such a break-point satisfies $\lambda_\circ = \frac{\polQ^{\tau'}_i-\polQ^\tau_i}{\pol^{\tau'}_i-\pol^\tau_i}$.

Hence, $(\lambda_\circ, 1)$ is a seed of $(\pol, \polQ)$ if and only if $\lambda_\circ > 0$ and $\lambda_\circ$ is a break-point of $\lambda\mapsto\b\phi(\lambda, 1)_i$ for some $i\in [n]$.
Note that $(\lambda, \mu)$ is a seed of $(\pol, \polQ)$ if and only $(\frac{\lambda}{\mu}, 1)$ is a seed of $(\pol, \polQ)$.
Thus, the pair $(\pol, \polQ)$ is fertile if and only if there exists $i\in [n]$ such that $\lambda\mapsto\b\phi(\lambda, 1)_i$ has a (strictly) positive break-point.

Pick some fixed $i \in [n]$.
By \Cref{lem:MinForPandQ}, there exists $\tau'\in \c S_n$ such that $\pol^{\tau'}_i = \max \pol^\sigma_i$ and $\polQ^{\tau'}_i = \max \polQ^\sigma_i$.
Thus, for any $\lambda \leq 0$, we have $ \argmin (\lambda\pol^\sigma_i - \polQ^\sigma_i) = \bigl(\argmax \pol^\sigma_i\bigr) \cap \bigl(\argmax\polQ^\sigma_i\bigr)$.
Consequently, a negative break-point of $\b\phi(\lambda, 1)_i$ would be of the form $\lambda_\circ = \frac{\polQ^{\tau'}_i-\polQ^\tau_i}{\pol^{\tau'}_i-\pol^\tau_i}$ for some $\tau'\in \bigl(\argmax \pol^\sigma_i\bigr) \cap \bigl(\argmax\polQ^\sigma_i\bigr)$. However, this implies that $\lambda_\circ \geq 0$, a contradiction. Therefore, all break-points of $\b\phi(\lambda, 1)_i$ are non-negative.

We will conclude this proof by showing that $0$ is the only break-point of $\b\phi(\lambda, 1)_i$ if and only if there exists $\tau\in \c S_n$ such that $\pol^{\tau}_i = \min\pol^\sigma_i$ and $\polQ^{\tau}_i = \max\polQ^\sigma_i$.

As $\polQ$ is not a point, we can assume $\min\polQ^\sigma_i\ne \max\polQ^\sigma_i$.
If $\bigl(\argmin\pol^\sigma_i\bigr)\cap\bigl(\argmax\polQ^\sigma_i\bigr)\ne\emptyset$, then for any $\lambda\geq 0$, we get: $\argmin (\lambda\pol^\sigma_i - \polQ^\sigma_i) = \bigl(\argmin\pol^\sigma_i\bigr)\cap\bigl(\argmax\polQ^\sigma_i\bigr)$.
This implies that $\b\phi(\lambda, 1)_i$ is affine in the interval $[0,\infty)$, and thus that the only break-point of $\b\phi(\lambda, 1)_i$ is $0$.

Reciprocally, if $\bigl(\argmin\pol^\sigma_i\bigr)\cap\bigl(\argmax\polQ^\sigma_i\bigr) = \emptyset$, 
consider $\tau\in \argmin\pol^\sigma_i$ which maximizes the value of $\polQ^\sigma_i$ (\ie $\tau \in \argmax_{\sigma\in \argmin\pol^\sigma_i} \polQ^\sigma_i$).
For any $\lambda$ large enough, $\tau\in\argmin (\lambda\pol^\sigma_i - \polQ^\sigma_i)$.
However, $\tau\notin \argmin ( - \polQ^\sigma_i)$ because $\polQ^\tau_i \ne\max\polQ^\sigma_i$.
Consequently, the two infinite intervals on which $\b\phi(\lambda, 1)_i$  is affine do not intersect at $0$: either they intersect at a strictly positive break-point; or there are at least two break-points, one of which being strictly positive.		
\end{proof}

\begin{example}\label{ex:fertile}
We shortly illustrate the use and efficiency of \Cref{thm:NotFertile} by concluding \Cref{exm:FertilitySmallExample}.
For $\pol = \polS_1$ and $\polQ = \polT_1$, we draw in \Cref{fig:MinMaxS1T1}, for each $i\in [3]$ and each $\sigma\in \c S_3$, the point $\bigl(\pol_i^\sigma, \polQ_i^\sigma\bigr)$.
For each $i\in [3]$, \Cref{lem:MinForPandQ} ensures that there is a top-right-most and a bottom-left-most point in the drawing.
For some $i\in [3]$, a permutation $\tau\in \c S_n$ is such that $\pol^\tau_i = \min_{\sigma\in \c S_n} \pol^\sigma_i$ and $\polQ^\tau_i = \max_{\sigma\in \c S_n} \polQ^\sigma_i$ if and only if the corresponding point is in the ``top-left corner'' of the drawing for $i$ (\ie both among the top-most points and among the left-most points).
As there is no such point for $i = 1$ (nor for $i= 2$), \Cref{thm:NotFertile} ensures that $(\polS_1, \polT_1)$ is fertile.

On the opposite, there is a point on the bottom-left corner for each $i\in [3]$, hence \Cref{thm:NotFertile} implies that $(\polT_1, \polS_1)$ is not fertile.
\end{example}

We are finally ready to prove the heredity of fertility which we left open earlier.

\begin{proof}[Proof of \Cref{thm:FertilityIsHeriditary}]
Suppose $\bigl(\liftGPsum[\polS][\polR],\, \liftGPsum\bigr)$ is not fertile.
By \Cref{thm:NotFertile}, there exists $\tau\in \c S_{n+1}$ such that $(\liftGPsum[\polS][\polR])^\tau_i = \min_{\sigma\in \c S_{n+1}} (\liftGPsum[\polS][\polR])^\sigma_i$ and $(\liftGPsum)^\tau_i = \max_{\sigma\in \c S_{n+1}} (\liftGPsum)^\sigma_i$.
By \Cref{lem:MaximizerInLowerDimension}, $\pol^{\tau^\downarrow}_i = \max_{\sigma\in \c S_n} \pol^\sigma_i$, and, symmetrically $\polS^{\tau^\downarrow}_i = \min_{\sigma\in \c S_n} \polS^\sigma_i$.
Again by \Cref{thm:NotFertile}, this implies that $(\polS, \pol)$ is not fertile, contradicting our assumption.
\end{proof}
To finish this section, we prove a new lower bound on the number of rays of the submodular cone $\SC$, and give a some-what explicit way to construct these vertices.
If $(\pol, \polQ)$ is fertile, then for $\lambda_\circ = \min\{\lambda\,;\,(\lambda, 1)\in \Lambda(\pol, \polQ)\}$, we denote $\mathdefn{\liftGPsumLM} \coloneqq\; \liftGPsum[\bigl(\polQ + \b\phi(\lambda_\circ,1)\bigr)][\bigl(\lambda_\circ\pol\bigr)]$.
Note that, by \Cref{prop:SeedIsRay}, the polytope $\liftGPsumLM\in\SC[n+1]$ is indecomposable.

\begin{corollary}\label{cor:CollectionOfRays}
If $\{\pol_1, \pol_2, \dots, \pol_r\}$ is a collection of indecomposable deformed $n$-permutahedra that are pairwise fertile, \ie such that the pair $(\pol_i, \pol_j)$ is fertile for all $i\in [r]$ and $j\in [r]$, then $\{\liftGPsumLM[\pol_j][\pol_i] ~;~ 1\leq i, j\leq r\}$ is a collection of indecomposable deformed $(n+1)$-permutahedra which are also pairwise fertile.
\end{corollary}

\begin{proof}
Apply \Cref{thm:FertilityIsHeriditary} to each pair $(\liftGPsumLM[\pol_j][\pol_i],\, \liftGPsumLM[\pol_\ell][\pol_k])$ for $1\leq i, j, k, \ell\leq n$.
\end{proof}

\begin{corollary}\label{cor:LowerBoundNumberOfRays}
Let $\mathdefn{t_n}$ be the number of rays of the cone $\SC$, and let $\mathdefn{p_k}$ be the maximal size of a family of pairwise fertile rays of $\SC[k]$, then we have:
$$t_n \geq p_k^{2^{n-k}}.$$
\end{corollary}

\begin{proof}
Let $\c C_k$ be a family of pairwise fertile rays of $\SC[k]$, with $|\c C_k| = p_k$.
For each $n\geq k$, let $\c C_{n+1} = \{\liftGPsumLM ~;~ \pol\in\c C_n,\,\polQ\in\c C_n\}$.
By \Cref{cor:CollectionOfRays}, all polytopes in $\c C_n$ are indecomposable deformed $n$-permutahedra.
Moreover, the bottom and top faces of $\liftGPsumLM$ are (isomorphic to) $\pol$ and $\polQ$ respectively, because $\lambda, \mu > 0$ in a seed.
Hence, if no two polytopes in $\c C_n$ are normally equivalent, then no two polytopes in $\c C_{n+1}$ are normally equivalent.
Thus, the number of rays of $\SC$ is lower bounded by~$|\c C_n|$.
We have: $|\c C_{n+1}| = |\c C_n|^2$;
consequently, $|\c C_n| = p_k^{2^{n-k}}$.
\end{proof}

In the upcoming \Cref{exm:FamillyOfPairwiseFertileN4,exm:FamillyOfPairwiseFertileN5} we will provide explicit families of pairwise fertile rays of $\SC[k]$ for the cases $k=4$ and $k=5$, which combined with the previous corollary give the following bounds. (We provide both bounds because, even if the one arising from $k=5$ is larger, the first is easier to check ``by hand'' as it only concerns an explicit family of 11 polytopes that we can list in the paper.)

\begin{corollary}\label{cor:ExplicitLowerBoundNumberOfRays}
The number $\mathdefn{t_n}$ of rays of $\SC$ satisfies
 $t_n \geq 11^{2^{n-4}}$ for all $n\geq 4$ and $t_n \geq 656^{2^{n-5}}$ for all $n\geq 5$.
\end{corollary}

\begin{proof}
Combine \Cref{cor:LowerBoundNumberOfRays} with \Cref{exm:FamillyOfPairwiseFertileN4,exm:FamillyOfPairwiseFertileN5}.
\end{proof}

\begin{example}\label{exm:FamillyOfPairwiseFertileN4}
The following collection of 11 polytopes $\{\pol_1, \dots, \pol_{11}\}$ are:
\begin{compactenum}
\item[a.] deformed $4$-permutahedra (\ie all the edges are in directions $\b e_i - \b e_j$);
\item[b.] indecomposable polytopes;
\item[c.] pairwise not normally equivalent (careful: they can be rotations of one another);
\item[d.] pairwise fertile (\ie for any $1\leq i, j\leq 11$, both $(\pol_i, \pol_j)$ and $(\pol_j, \pol_i)$ are fertile).
\end{compactenum}

Each polytope is given as a list of vertices, see \Cref{fig:4Permutahedra}.
Checking the direction of the edges is straightforward.
The indecomposability has been shown in \cite{PadrolPoullot-NewCriteriaIndecomposability}, among others, or can be checked by computing the deformation cone (and observing that it is a ray).
The non-normal-equivalence can be checked by computing, for each polytope, the subsets of $\{\pol^\sigma = \b v ~;~ \sigma\in \c S_4\}$ for each vertex of $\pol$; then comparing these lists between two polytopes.
The pairwise fertility can be checked using \Cref{thm:NotFertile}.

\noindent$\pol_{1} \coloneqq \conv\Bigl($
(2, 2, 3, 1),
(2, 3, 2, 1),
(2, 3, 3, 0),
(3, 2, 2, 1),
(3, 2, 3, 0),
(3, 3, 2, 0)
$\Bigr)$

\noindent$\pol_2 \coloneqq \conv\Bigl($
(2, 2, 3, 1),
(2, 3, 2, 1),
(2, 3, 3, 0),
(3, 2, 2, 1),
(3, 2, 3, 0),
(3, 3, 2, 0),
(4, 2, 2, 0)
$\Bigr)$

\noindent$\pol_3 \coloneqq \conv\Bigl($
(2, 2, 3, 1),
(2, 3, 2, 1),
(2, 3, 3, 0),
(2, 4, 2, 0),
(3, 2, 2, 1),
(3, 2, 3, 0),
(3, 3, 2, 0)
$\Bigr)$

\noindent$\pol_4 \coloneqq \conv\Bigl($
(2, 2, 3, 1),
(2, 2, 4, 0),
(2, 3, 2, 1),
(2, 3, 3, 0),
(3, 2, 2, 1),
(3, 2, 3, 0),
(3, 3, 2, 0)
$\Bigr)$

\noindent$\pol_5 \coloneqq \conv\Bigl($
(2, 2, 2, 2),
(2, 2, 3, 1),
(2, 3, 2, 1),
(2, 3, 3, 0),
(3, 2, 2, 1),
(3, 2, 3, 0),
(3, 3, 2, 0)
$\Bigr)$

\noindent$\pol_6 \coloneqq \conv\Bigl($
(4, 5, 5, 2),
(4, 5, 6, 1),
(4, 6, 4, 2),
(4, 6, 5, 1),
(5, 4, 5, 2),
(5, 4, 6, 1),
(5, 5, 4, 2),
(5, 5, 6, 0),
(5, 6, 4, 1),
(5, 6, 5, 0),
(6, 4, 5, 1),
(6, 5, 4, 1),
(6, 5, 5, 0)
$\Bigr)$

\noindent$\pol_7 \coloneqq \conv\Bigl($
(4, 4, 6, 2),
(4, 5, 5, 2),
(4, 5, 6, 1),
(4, 6, 5, 1),
(5, 4, 5, 2),
(5, 4, 6, 1),
(5, 5, 4, 2),
(5, 5, 6, 0),
(5, 6, 4, 1),
(5, 6, 5, 0),
(6, 4, 5, 1),
(6, 5, 4, 1),
(6, 5, 5, 0)
$\Bigr)$

\noindent$\pol_8 \coloneqq \conv\Bigl($
(4, 5, 5, 2),
(4, 5, 6, 1),
(4, 6, 5, 1),
(4, 6, 6, 0),
(5, 4, 5, 2),
(5, 4, 6, 1),
(5, 5, 4, 2),
(5, 5, 6, 0),
(5, 6, 4, 1),
(5, 6, 5, 0),
(6, 4, 5, 1),
(6, 5, 4, 1),
(6, 5, 5, 0)
$\Bigr)$

\noindent$\pol_9 \coloneqq \conv\Bigl($
(4, 5, 5, 2),
(4, 5, 6, 1),
(4, 6, 5, 1),
(5, 4, 5, 2),
(5, 4, 6, 1),
(5, 5, 4, 2),
(5, 5, 6, 0),
(5, 6, 4, 1),
(5, 6, 5, 0),
(6, 4, 4, 2),
(6, 4, 5, 1),
(6, 5, 4, 1),
(6, 5, 5, 0)
$\Bigr)$

\noindent$\pol_{10} \coloneqq \conv\Bigl($
(4, 5, 5, 2),
(4, 5, 6, 1),
(4, 6, 5, 1),
(5, 4, 5, 2),
(5, 4, 6, 1),
(5, 5, 4, 2),
(5, 5, 6, 0),
(5, 6, 4, 1),
(5, 6, 5, 0),
(6, 4, 5, 1),
(6, 4, 6, 0),
(6, 5, 4, 1),
(6, 5, 5, 0)
$\Bigr)$

\noindent$\pol_{11} \coloneqq \conv\Bigl($
(4, 5, 5, 2),
(4, 5, 6, 1),
(4, 6, 5, 1),
(5, 4, 5, 2),
(5, 4, 6, 1),
(5, 5, 4, 2),
(5, 5, 6, 0),
(5, 6, 4, 1),
(5, 6, 5, 0),
(6, 4, 5, 1),
(6, 5, 4, 1),
(6, 5, 5, 0),
(6, 6, 4, 0)
$\Bigr)$

\begin{figure}
    \centering
    \begin{subfigure}[b]{0.18\linewidth}
         \centering
         \includegraphics[width=\textwidth]{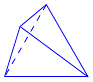}
         \caption{Tetrahedron}
         \label{sfig:Tetrahedron}
    \end{subfigure}
    \begin{subfigure}[b]{0.16\linewidth}
         \centering
         \includegraphics[width=\textwidth]{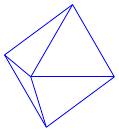}
         \caption{Square pyr.}
         \label{sfig:Pyramid}
    \end{subfigure}
    \begin{subfigure}[b]{0.16\linewidth}
         \centering
         \includegraphics[width=\textwidth]{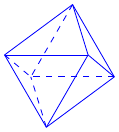}
         \caption{Octahedron}
         \label{sfig:Octahedron}
    \end{subfigure}
    \begin{subfigure}[b]{0.17\linewidth}
         \centering
         \includegraphics[width=\textwidth]{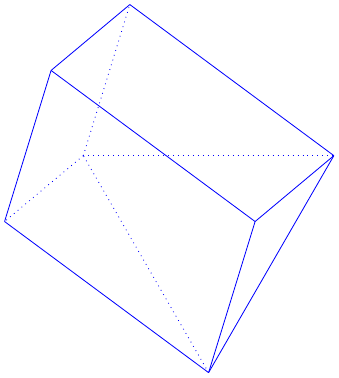}
         \caption{\emph{Strawberry}}
         \label{sfig:Strawberry}
    \end{subfigure}
    \begin{subfigure}[b]{0.22\linewidth}
         \centering
         \includegraphics[width=\textwidth]{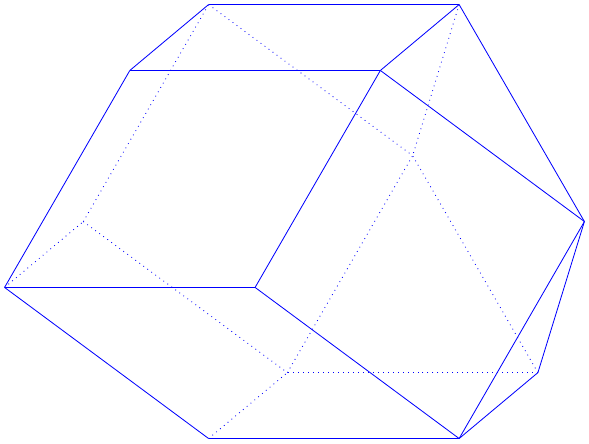}
         \caption{\emph{Persimmon}}
         \label{sfig:Persimmon}
    \end{subfigure}
    \caption[Indecomposable deformed $4$-permutahedra]{
    Up to permutation of the coordinates and central symmetry, each (polytope corresponding to a) ray of $\SC[4]$ is one of these five polytopes (using the names from \cite{PadrolPoullot-NewCriteriaIndecomposability}) or a triangle, or a segment.
    Accounting for all the possible permutations of coordinates and central symmetry, the cone $\SC[4]$ has 37 rays. In \Cref{exm:FamillyOfPairwiseFertileN4}, $\pol_i$ is an octahedron for $i = 1$;  a \emph{Strawberry} for $2\leq i\leq 5$; and a \emph{Persimmon} for $6\leq i\leq 11$.}
    \label{fig:4Permutahedra}
\end{figure}
\end{example}

\begin{example}\label{exm:FamillyOfPairwiseFertileN5}
For $n = 5$, there exists a family of 656 polytopes $\{\pol_1, \dots, \pol_{656}\}$ which are:
\begin{compactenum}
\item[a.] deformed $5$-permutahedra (\ie all the edges are in directions $\b e_i - \b e_j$);
\item[b.] indecomposable polytopes;
\item[c.] pairwise not combinatorially equivalent (\ie having pairwise non-isomorphic face lattices, see \cite[below Theorem 2.7]{Ziegler-polytopes});
\item[d.] pairwise fertile (\ie for any $1\leq i, j\leq 656$, both $(\pol_i, \pol_j)$ and $(\pol_j, \pol_i)$ are fertile).
\end{compactenum}

The full list is available as supplementary material at \cite{EmbrunForestier}.
To obtain this list of polytopes, we first listed all the $117\,978$ deformed $5$-permutahedra (by computing the submodular cone $\SC[5]$ via its facet description then computing its rays\footnote{One can also use the Csirmaz--Csirmaz database~\cite{CsirmazCsirmaz-AttemptingTheImpossible} to reproduce our computations.}), and kept one representative of each of the 672 classes of combinatorially equivalent deformed $5$-permutahedra.
Then, we defined a graph whose vertices are these 672 polytopes, and where there is an edge between $\pol$ and $\polQ$ if both $(\pol, \polQ)$ and $(\polQ, \pol)$ are fertile (which can be checked quickly thanks to \Cref{thm:NotFertile}).
A clique in this graph is a family satisfying the above criteria a. to d.: we extracted a clique of maximal size, gathering 656 polytopes.

Note that this algorithm might produce different results depending on the chosen representatives of the 672 classes of combinatorially equivalent deformed 5-permutahedra.

This gives the lower bound $p_5\geq 656$.
Computing $p_5$ precisely would require to run a fertility test on all pairs of $117\,978$ deformed $5$-permutahedra, then find the maximal clique in the resulting graph (with $117\,978$ vertices).
We did not run this computation.
\end{example}

According to the previous example, there are 672 combinatorially-different indecomposable deformed $5$-permutahedra, and 656 of them which form a family of pairwise fertile pairs.
This motivates the following conjecture.

\begin{conjecture}\label{conj:RatioFertility}
Let $t_n$ be the number of rays of the submodular cone $\SC$, and let $u_n$ be the maximum $s$ such that there exists a family $\{\pol_1, \dots, \pol_s\}$ of indecomposable deformed $n$-permutahedra such that both $(\pol_i, \pol_j)$ and $(\pol_j, \pol_i)$ are fertile, for all $i, j\in[s]$, and $\pol_i, \pol_j$ are not normally equivalent for $i\ne j$.
We conjecture that $\frac{u_n}{t_n}\to 1$ when $n\to +\infty$.
\end{conjecture}

\subsection{Upper and lower bounds on the number of rays of the submodular cone}\label{ssec:NewRays}

The number of rays of $\SC[2]$, $\SC[3]$, $\SC[4]$ and $\SC[5]$ are well-known:
for $n \in\{2, 3, 4, 5\}$, the cone $\SC$ has respectively $1$, $6$, $37$, $117\, 978$ rays.
For $n = 6$, this number is estimated between $10^{13}$ and $10^{20}$ by recent work of Csirmaz and Csirmaz \cite{CsirmazCsirmaz-AttemptingTheImpossible} who computed more than $360$ billion rays of~$\SC[6]$ (and explained why obtaining the exact count shall be regarded as ``impossible'', already for $n = 6$).
To give context to our bound from \Cref{cor:ExplicitLowerBoundNumberOfRays}, we compare it to the current best upper and lower bounds for the number of rays of $\SC$ that we are aware of, reported in~\cite{HaimanYao:2023}.

\begin{definition}
Let $\mathdefn{t_n}$ be the number of rays of $\SC$ and $\mathdefn{a_n}$ be the number of rays of $\SC$ containing polytopes $\pol$ such that both $\pol^{+\b e_n}$ and $\pol^{-\b e_n}$ are indecomposable (\ie rays of $\SC[n-1]$).
 \end{definition}

As far as we are aware, the only known upper bound for $t_n$ was found in \cite[Theorem~5.1]{HaimanYao:2023}. We can slightly improve it (roughly to its square root) by using McMullen's Upper Bound Theorem \cite{McMullen-MaximumNumberOfFaces,Stanley-NumberOfFaces}, which gives a tight bound for the number of vertices of a polytope in terms of its number of facets. For our asymptotic applications, we only need the following simpler bound, that has a very simple and elegant proof in \cite[Proposition 5.5.2]{Matousek}.

\begin{theorem}[{Asymptotic Upper Bound Theorem \cite[Prop. 5.5.2]{Matousek}}]\label{thm:UBTgeneral}
For any $d$-polytope with $f_{d-1}$ facets, its number $f_0$ of vertices satisfies:
$$f_0 \leq 2\,\binom{f_{d-1}}{\lfloor \frac{d}{2}\rfloor}$$
\end{theorem}

\begin{proposition}\label{prop:upperbound}
For all $n\geq 4$, we have:
$t_n \,\leq\, n^{2^n}$.
\end{proposition}

\begin{proof}
We work out the upper bound as an application of the Asymptotic Upper Bound Theorem. Note that when we talk about rays of $\SC$, we are actually considering its quotient modulo its lineality.
The pointed cone $\SC/\lineal(\SC)$ has dimension $2^n-n-1$. It is well known that it has $\binom{n}{2}2^{n-2}$ facets, which are given by the submodular inequalities of the form
\[f(X \cup y) - f(X) \geq f(X \cup \{y,z\}) - f(X \cup z),\]
as shown for example in \cite[Thm.~7.2]{ArdilaCastilloEurPostinkov2020-CoxeterSubmodular}.
See \Cref{fig:SC3} for an illustration of the case $n=3$.
Hence it is a cone over a polytope of dimension $2^n-n-2$ with $n(n-1)2^{n-3}$ facets (and the rays of $\SC$ are in bijection with vertices of this polytope).
By the application of the Asymptotic Upper Bound \Cref{thm:UBTgeneral}, and the usual inequality $\binom{a}{b}\leq \left(\frac{a}{b}\right)^b e^b$, we get that:
$$t_n \,\leq\, 2\,\binom{n(n-1)2^{n-3}}{2^{n-1} - \frac{n}{2}} \,\leq\, 2\,\binom{n(n-1)2^{n-3}}{2^{n-1}} \,\leq\, 2\,\left(\frac{n^2 2^{n-3}}{2^{n-1}}\right)^{2^{n-1}}e^{2^{n-1}} \,\leq\, 2\, \left(\frac{e^{1/2}}{4}\, n\right)^{2^{n}}.$$
As $\frac{e^{1/2}}{4} \approx 0.41$, we have $t_n \leq n^{2^n}$ for all $n\geq 4$, which gives the desired upper bound.
\end{proof}

Our inductive lower bound on the number of rays of the submodular cone in \cref{cor:ExplicitLowerBoundNumberOfRays} shows that $ t_n\geq a_n\geq 656^{2^{n-5}}$. Compared with the upper bound above gives:

\begin{theorem}\label{thm:UpperLowerBoundsTnSn}
For all $n\geq 4$:
$$\begin{array}{cccccccr}
     2^{2^{n-2}} &\leq& a_n &\leq& t_n &\leq& n^{2^{n}}&\text{ or equivalently }\\
     n-2 &\leq& \log_2\log_2 a_n &\leq& \log_2\log_2 t_n &\leq& n + \log_2\log_2 n 
\end{array}$$
\end{theorem}

Written in this form, we see that there is still room for improvement, but as we will discuss now, our contribution represents significative improvement with respect to the previous bounds.
As far as we are aware, the previous best known lower bound for $t_n$ was actually a lower bound on 
$\mathdefn{m_n}$, which is the number of rays of $\SC$ which contain $0/1$-polytopes (\ie the vertices of the polytope are in $\{0,1\}^n$).
Indeed, a $0/1$ deformed $n$-permutahedron is precisely the matroid polytope defined as $\pol_{\c M} = \conv\{\b e_B ~;~ B \text{ is a basis of } \c M\}$ for some matroid $\c M$ on the ground set $[n]$.
According to \cite[Proposition 10.4.12]{Nguyen1986-SemimodularFunctions}, the \textbf{indecomposable} $0/1$ deformed $n$-permutahedra are in bijection with the \textbf{connected} matroids on the ground set $[n]$: hence $m_n$ is the number of connected matroids.
A lower bound on the number of matroids was established by Knuth \cite{Knuth-NumberOfMatroids}, while Mayhew--Newman--Welsh--Whittle \cite{MayhewNewmanWelshWhittle-NumberConnectedMatroids} proved that the number of connected matroids being asymptotically at least $\frac{1}{2}$ of the total number of matroids (see \cite{BansalPendavinghVanDerPol2015-NumberMatroids} for a presentation of the literature on the bounds of the number of matroids). We use the following upper bound on the number of matroids from \cite{BansalPendavinghVanDerPol2015-NumberMatroids}.

\begin{theorem}[{\cite[Thm.~1]{BansalPendavinghVanDerPol2015-NumberMatroids}}]\label{thm:BoundsMn}
	For $n\to +\infty$, we have: %
	$$\log_2\log_2 m_n \leq n - \frac{3}{2}\log_2 n + \frac{1}{2}\log_2 \frac{2}{\pi} + 1 + o(1).$$
\end{theorem}

\begin{proof}
	By \cite[Prop.~10.4.12]{Nguyen1986-SemimodularFunctions}, the indecomposable $0/1$ deformed $n$-permutahedra are in bijection with the connected matroids on the ground set $[n]$.
	By \cite[Theorem 1]{BansalPendavinghVanDerPol2015-NumberMatroids}, the number of matroids (and hence of the connected ones) is upper bounded by the claimed bound.
\end{proof}

Note that this is smaller than $2^{c\cdot 2^n\cdot n^{-3/2}} \leq t_n$ for some constant $c > 0$.
Consequently, taking an indecomposable deformed $n$-permutahedron ``at random'', the probability that it is possible to write it with $0/1$-coordinate goes to $0$ as $n\to+\infty$.

\begin{corollary}
    When $n\to+\infty$, we have $\frac{m_n}{a_n} \to 0$ and $\frac{m_n}{t_n}\to 0$.
\end{corollary}
\begin{proof}
Using \Cref{thm:BoundsMn}, we have that $\log_2 \log_2 m_n - \log_2\log_2 a_n < 0$ for $n$ large enough, which implies the second statement.
\end{proof}

\begin{remark}
Note that we clearly have the strict inequality $a_n < t_n$, as already some rays of $\SC[4]$ have top and bottom faces which are both decomposable: according to \Cref{exm:FertilitySmallExample}, only segments, triangles, square pyramids, and octahedra can be constructed from pairs of rays of $\SC[3]$, hence the \emph{Strawberry} and \emph{Persimmon} from \Cref{fig:4Permutahedra} cannot.
\end{remark}

\begin{remark}
In our inductive construction of rays of $\SC[n+1]$ from pairs of rays of $\SC$, we proved that every seed can be used to construct a ray of $\SC[n+1]$.
Yet, we only used the fact that there exists at least 1 seed for each fertile pair (there might be many), and only used families of pairwise fertile pairs (so we used only a clique instead of the whole fertility graph).
Denoting $b_n$ the number of rays captured by the method of \Cref{ssec:NewRays}, we showed $b_{n+1} \geq b_n^2$; combining with \Cref{ssec:IndependencePolytopes}, we may improve to $b_{n+1} \geq b_n^2 + 2 b_n$. We did not pursue this path because it complicates the exposition without giving any asymptotic improvement to the lower bound.

\end{remark}
This raises the natural following open questions:

\begin{question}
What is the asymptotic behavior of the ratio $\frac{a_n}{t_n}$? %
\end{question}

\begin{question}
For a fertile pair $(\pol, \polQ)\in \SC^2$ (\ie $|\Lambda(\pol, \polQ)|\geq 1$), what is the \emph{typical} size of the seed-set $\Lambda(\pol, \polQ)$?
\end{question}

Denoting $c_k$ the average cardinal of $\Lambda(\pol, \polQ)$ for fertile pairs in $\SC[k]$, one would need that $\sum_{k = 0}^n \frac{\log_2 c_k}{2^k} \to +\infty$ when $n\to +\infty$, in order to hope to improve the order of magnitude of the exponent in our lower bound: using our current method,
this seems unlikely.

\appendix 

\section{Beyond rays: Inductive construction of the face lattice of $\SC$}\label{sec:equalitysets}
In this appendix, we push our inductive method further in order to construct the whole face lattice of $\SC$.
Recall that $\pol\normeq \pol'$ if $\pol$ and $\pol'$ are \defn{normally equivalent}, \ie if $\pol\deformed\pol'$ and $\pol'\deformed \pol$.

\subsection{Equality sets and admissibility cones}

For $\polR, \polS\in \SC[n+1]$, if $\polR\normeq\polS$, then $\polR^{+\b e_{n+1}}\normeq\polS^{+\b e_{n+1}}$ and $\polR^{-\b e_{n+1}}\normeq\polS^{-\b e_{n+1}}$.
However, %
in \Cref{fig:VariousGPsums}, the polytopes $\pol_1, \dots, \pol_4$ are normally equivalent, but the polytopes $\GPsum[\polQ][\pol_1],\dots, \GPsum[\polQ][\pol_4]$ are not.
This motivates the definition of equality sets, in order to store the additional information needed for retrieving the normal equivalent class of $\GPsum$ from the normal equivalent classes of $\pol$ and $\polQ$, see \Cref{fig:VariousGPsums,fig:SC3,fig:SC3_generic} for illustrations.

\begin{definition}\label{def:EqualitySet}
For $\pol, \polQ\in \SC$ their \defn{equality set} is
$\mathdefn{\b E(\pol, \polQ)} \coloneqq \{(i, \sigma)\in [n]\times \c S_n ~;~ \polQ_i^\sigma = \pol_i^\sigma\}$, and their \defn{poset of valid equality sets} is:
$$\mathdefn{\c E(\pol, \polQ)} \coloneqq \{\b E(\pol', \polQ') ~;~ \pol', \polQ'\in \SC,\, \pol'\sim \pol \text{ and } \polQ'\sim \polQ\}.$$
\end{definition}

\begin{remark}\label{rmk:EmptyIsMax}
We will order the poset of valid equality sets by reverse inclusion $\supseteq$.
By \Cref{rmk:GenericGPsum}, we have: $\emptyset\in \c E(\pol, \polQ)$ for all $\pol, \polQ\in \SC$; hence $\emptyset$ is always the maximum of $\c E(\pol, \polQ)$.
\end{remark}

The main object of this appendix allowing us to study the faces of $\DefoCone[\GPsum]$ is:

\begin{definition}
The \defn{$\b E$-admissibility cone} of $\pol, \polQ\in\SC$ with $\b E \in \c E(\pol, \polQ)$ is:
$$\mathdefn{\admissconeE} \coloneqq \bigl\{\GPsum[\polQ'][\pol'] ~;~ \pol'\deformed\pol,~ \polQ'\deformed\polQ, \text{ satisfying } \b E(\pol', \polQ') \supseteq \b E\bigr\}.$$

\end{definition}

Note that the $\emptyset$-admissibility cone is the usual admissibility cone.
In the following, we extend the results of \Cref{ssec:AdmissCones} to $ \b E$-admissibility cones.
First, we generalize \Cref{prop:AdmissConeIsPolyhedral}.

\begin{proposition}\label{prop:EAdmissConeIsPolyhedral}
	For $\pol, \polQ\in \SC$, the cone $\admissconeE$ is a polyhedral cone which is linearly isomorphic to the cone $\bigl(\DefoCone[\polQ] \times \DefoCone\bigr) \cap \bigl\{ (\polQ, \pol) ~;~ \pol,\polQ\in \SC, \text{ with } \b E(\pol, \polQ) \supseteq \b E\bigr\}$.
\end{proposition}

\begin{proof}
The cone $\bigl\{ (\polQ',\pol')\in \SC^2 ~;~ \b E(\pol', \polQ') \supseteq \b E\bigr\}$ is polyhedral, because it is the intersection of $\SC^2$ with the hyperplanes $(\pol')_i^\sigma=(\polQ')_i^\sigma$ for $(i,\sigma)\in \b E$, and the half-spaces of the form $(\pol')_i^\sigma\leq(\polQ')_i^\sigma$ for $(i,\sigma)\notin \b E$.
We conclude by noting that $\admissconeE$ is the image of this polyhedral cone under the linear map $(\polQ,\pol) \mapsto \GPsum$ (see \Cref{lem:GPsumIsLinear}).
\end{proof}

Now we generalize \Cref{thm:AdmissibilityConeIsDeformationCone}. To do so, we need the following lemma:

\begin{lemma}\label{lem:EPQintersections}
	For $\pol_1,\pol_2, \polQ_1, \polQ_2\in \SC$ with $\pol_1\normeq\pol_2$ and $\polQ_1\normeq\polQ_2$, if $\b E_1 \coloneqq \b E(\pol_1, \polQ_1)$ and $\b E_2 \coloneqq \b E(\pol_2, \polQ_2)$, then $\b E(\pol_1+\pol_2,\, \polQ_1 + \polQ_2) = \b E_1\cap\b E_2$.
	Consequently, for $\pol, \polQ\in \SC$, if $\b E_1, \b E_2\in \c E(\pol, \polQ)$, then
	$\b E_1\cap \b E_2\in \c E(\pol, \polQ)$.
\end{lemma}

\begin{proof}
	For any $(i, \sigma)\in \b E_1\cap \b E_2$, we have:
	$(\polQ_1+\polQ_2)_i^\sigma = (\polQ_1)_i^\sigma + (\polQ_2)_i^\sigma = (\pol_1)_i^\sigma + (\pol_2)_i^\sigma = (\pol_1+\pol_2)_i^\sigma $.
	Conversely, if $(i, \sigma)\notin \b E_1\cap \b E_2$, we have:
	$(\polQ_1+\polQ_2)_i^\sigma = (\polQ_1)_i^\sigma + (\polQ_2)_i^\sigma < (\pol_1)_i^\sigma + (\pol_2)_i^\sigma = (\pol_1+\pol_2)_i^\sigma $, where the strict inequality comes from the fact that one term is strictly smaller than its counterpart (and the other term weakly smaller).
	Hence $\b E(\pol_1+\pol_2,\, \polQ_1 + \polQ_2) = \b E_1\cap\b E_2$.
	
	The second fact is a consequence of the first, taking $\pol_1, \pol_2\normeq\pol$, and $\polQ_1, \polQ_2\normeq\polQ$. %
\end{proof}

\begin{figure}
	\centering
	\includegraphics[width=0.75\linewidth]{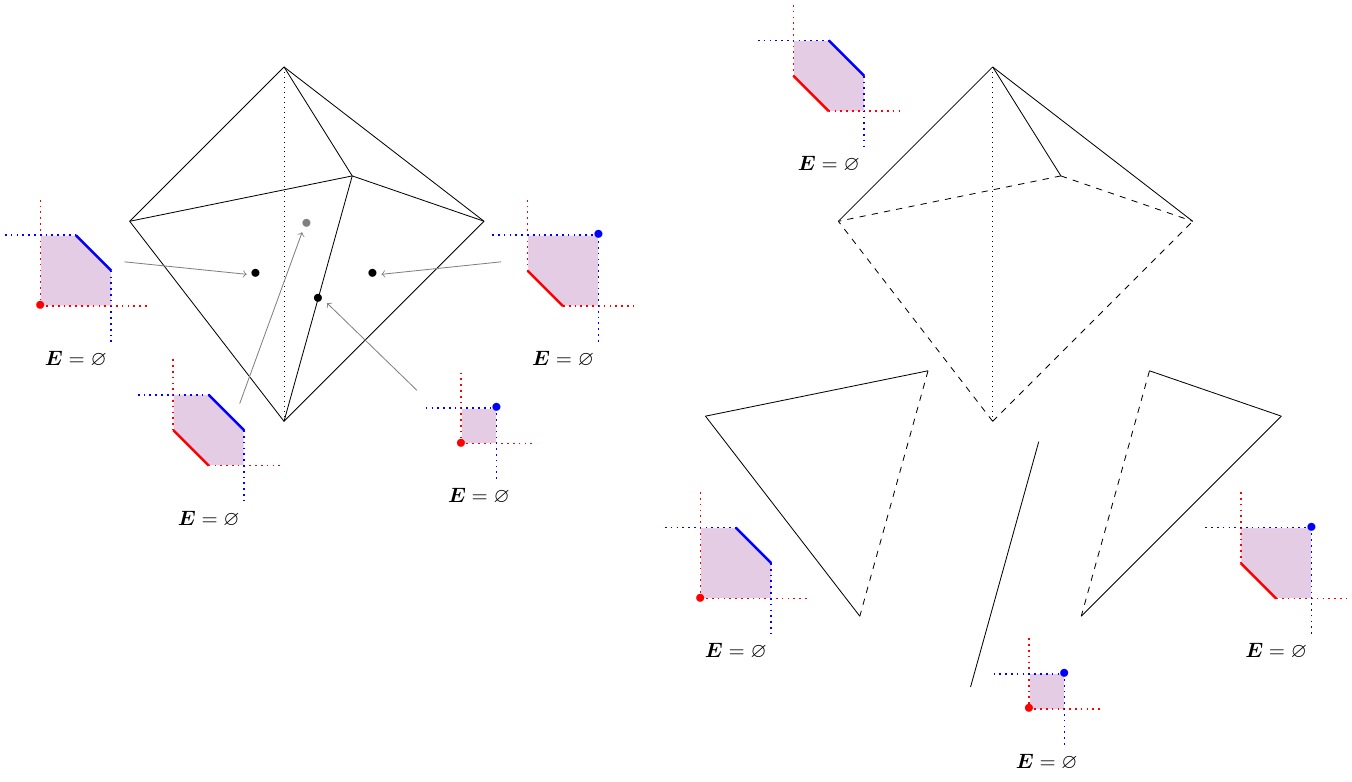}
	\caption[{Partition of $\SC[3]$ into lattices of equality sets}]{(Left) The four faces of $\SC[3]$ of the form $\DefoCone[\genericGPsum]$.
		(Right) The decomposition of $\SC[3]$ into four collections of faces.
		Each collection corresponds to the faces $\DefoCone[\GPsum]$ with a fixed class of normally equivalence for the top and bottom faces.
		For instance, the bottom right (half-open) triangle corresponds to the three faces $\DefoCone[\GPsum]$ where $\polQ$ is a segment and $\pol$ is a point.}
	\label{fig:SC3_generic}
\end{figure}

We now prove that the faces of $\SC[n+1]$ are precisely the $\b E$-admissibility cones 
parameterized by pairs of elements in $\SC$. 

\begin{theorem}\label{thm:EAdmissibilityConeIsDeformationCone}
	For $\pol, \polQ\in \SC$ with $\b E \coloneqq \b E(\pol, \polQ)$, we have:
	\begin{align*}
		\DefoCone[\GPsum]  &= \admissconeE&\text{ and }&&
		\DefoCone[\liftGPsum]  &= \admissconeE\oplus \R\b e_{n+1}.
	\end{align*}
\end{theorem}

\begin{proof}
	These two equalities are equivalent by \Cref{thm:bijections}.	
    We mimic the proof of \Cref{thm:AdmissibilityConeIsDeformationCone}.
	
	Let $\polR_0=\GPsum$ and choose representatives $\polR_1,\dots,\polR_s$ of all the other normal equivalence classes of polytopes in~$\admissconeE$.
    Let $\polS=\sum_{i=0}^s \polR_i$. 
    We will show $\admissconeE= \DefoCone[\polS]$, and $\polS\normeq\GPsum$. 

    Any $\polR\in \admissconeE$ is normally equivalent to some $\polR_i$, which is a Minkowski summand of $\polS$ by construction. Therefore, 
	$\polR\in \DefoCone[\polS]$ and thus $\admissconeE\subseteq \DefoCone[\polS]$.
    
	For the reverse inclusion, observe that $\polR_i\in \admissconeE$ implies the existence of $\pol_i\deformed\pol$ and $\polQ_i\deformed\polQ$ with $\polR_i=\GPsum[\pol_i][\polQ_i]$ and $\b E_i \coloneqq \b E(\pol_i, \polQ_i)$ for some $\b E_i \supseteq \b E$. 
    Hence, we get $\pol[S] = \sum_{0\leq i\leq s}(\GPsum[\pol_i][\polQ_i]) = \GPsum[\left(\sum_{0\leq i\leq s}\polQ_i\right)][\left(\sum_{0\leq i\leq s}\pol_i\right)]$.
	In particular $\polS^{+\b 1}=\sum_{0\leq i\leq s}\pol_i\normeq\pol$, because $\pol_0=\pol$ and $\pol_i\deformed\pol$ for all $i$, and similarly $\polS^{-\b 1}=\sum_{0\leq i\leq s}\polQ_i\normeq\polQ$.
	By 
	\Cref{lem:EPQintersections}, we get $\b E\left(\sum_{0\leq i\leq s}\pol_i,\,\sum_{0\leq i\leq s}\polQ_i\right) = \bigcap_{0\leq i\leq s} \b E(\pol_i,\polQ_i) = \bigcap_{0\leq i\leq s} \b E_i = \b E$, because $\b E_0=\b E$ and $\b E_i\supseteq \b E$ for all $i$.
	Thus, $\pol[S]\in \admissconeE$.%

    Now, let $\polR\in \DefoCone[\polS]$ be a deformation of $\polS$. 
	Then $\pol'\coloneqq \polR^{+\b 1}\deformed \polS^{+\b 1}\deformed \pol$ and $\polQ'\coloneqq \polR^{-\b 1}\deformed \polQ$. 
	To show that $\polR\in\admissconeE$, it remains to prove that $\b E(\pol', \polQ')\supseteq \b E$.
	Note that, as $\polR\deformed\polS$, there exists $\pol[T]\in \DefoCone[\polS]$ such that $\polR + \pol[T] = \polS$. Hence, by what we have just proven, $\pol[T] = \GPsum[\polQ''][\pol'']$ for some $\pol''\deformed\pol$ and $\polQ''\deformed\polQ$, implying that $\polS = \GPsum[(\polQ'+\polQ'')][(\pol' + \pol'')]$, by \Cref{lem:GPsumIsLinear}.
	Fix $(i, \sigma)\in \b E$.
	If $(\polQ')_i^\sigma< (\pol')_i^\sigma$, then $(\polQ'+\polQ'')_i^\sigma < (\pol'+\pol'')_i^\sigma$, contradicting the fact that $\polS\in \admissconeE$.
	Consequently, $(\polQ')_i^\sigma = (\pol')_i^\sigma$ for all $(i,\sigma)\in \b E$, and $\pol[R]\in \admissconeE$, proving: $\admissconeE = \DefoCone[\polS]$.
	
	It remains to show that $\polS\normeq\GPsum$. %
    By \Cref{prop:EAdmissConeIsPolyhedral}, the cone $\admissconeE$ is linearly isomorphic to $\bigl(\DefoCone[\polQ] \times \DefoCone\bigr) \cap \bigl\{ (\polQ, \pol) ~;~ \pol,\polQ\in \SC, \text{ with } \b E(\pol, \polQ) \supseteq \b E\bigr\}$ via $(\polQ,\pol)\mapsto \GPsum$.
    As it satisfies $\polQ^\sigma_i < \pol^\sigma_i$ for all $(i,\sigma) \notin \b E$, the pair $(\polQ,\pol)$ lies in the interior of this cone.
    Hence, the polytope $\GPsum$ lies in the interior of $\DefoCone[\polS] = \admissconeE$, implying $\polS \sim\GPsum$.
\end{proof}

We can use \Cref{thm:EAdmissibilityConeIsDeformationCone} to parameterize of the face lattice $\scr L(\SC[n+1])$ by $\scr L(\SC)\times\scr L(\SC)$.
Note that it implies that $\GPsum[\polQ'][\pol']\deformed\GPsum$ if and only if $\pol'\deformed\pol$, $\polQ'\deformed\polQ$, and $\b E'\supseteq\b E$.
Therefore, via the map $(\pol, \polQ, \b E) \mapsto \DefoCone[{\lift\!(\GPsum)}]$, the lattice $\scr L(\SC[n+1])$ is in bijection with the lattice of triples $(\pol, \polQ, \b E)$ with $\pol, \polQ\in \SC$ and $\b E\in \c E(\pol, \polQ)$, ordered by the relations:
\begin{compactitem}
\item $(\pol', \polQ', \b E')\deformed(\pol, \polQ, \b E)$ if and only if $\pol'\deformed\pol$, $\polQ'\deformed\polQ$, and $\b E'\supseteq\b E$; and
\item $(\pol', \polQ', \b E')\normeq(\pol, \polQ, \b E)$ if and only if $\pol'\normeq\pol$, $\polQ'\normeq\polQ$ and $\b E'=\b E$. 
\end{compactitem}

This bijection hints at studying two subposets of $\scr L(\SC[n+1])$:
for fixed $\pol, \polQ$, the poset $\c E(\pol, \polQ)$; and for fixed $\b E$, the poset of $(\pol, \polQ)\in \SC^2$ with $\b E(\pol, \polQ) = \b E$.
The latter is of special interest when $\b E = \emptyset$, as we discuss in the next subsection.

\begin{figure}
    \centering
    \includegraphics[width=0.8\linewidth]{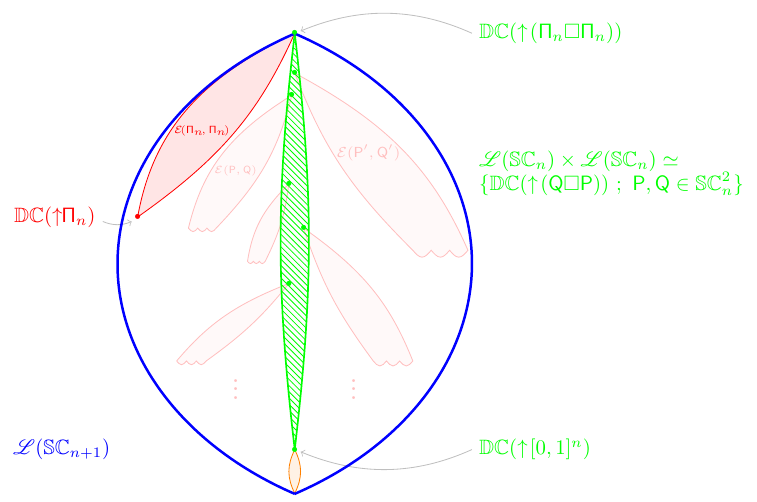}
    \caption{The lattice {\color{blue} $\scr L(\SC[n+1])$} is partitioned into join-semi-lattices {\color{pink} $\c E(\pol, \polQ)$} for $\pol, \polQ\in \SC^2$, each having a {\color{green} unique maximum} associated to the generic GP-sum $\genericGPsum$, but possibly several minima.
    The set of all generic GP-sums is the face figure of the (face associated to) the cube $[0, 1]^n$.
    This face figure is isomorphic to $\scr L(\SC)\times \scr L(\SC)$.
    Besides, the face associated to the cube $[0, 1]^n$ is itself a simplicial cone, so {\color{orange} $\c E(\b 0, \b 1)$} is a Boolean lattice.}
    \label{fig:FaceLatticeVsEqualitySets}
\end{figure}

\subsection{Generic GP-sums}
For fix $\pol,\polQ\in\SC$, consider the subposet of faces of $\SC[n+1]$ of the form $\DefoCone[{\lift\!(\GPsum[\polQ'][\pol'])}]$ for some $\pol'\normeq \pol$ and $\polQ'\normeq \polQ$.
By \Cref{thm:EAdmissibilityConeIsDeformationCone}, it is in bijection with the poset of valid equality sets $\c E(\pol, \polQ)$.
As said in \Cref{rmk:EmptyIsMax}, $\emptyset$ is always the maximum of $\c E(\pol, \polQ)$, independently of $\pol, \polQ\in \SC$, which motivates the following definition.

\begin{definition}
A \defn{generic GP-sum} of $\pol$ and $\polQ$ is: $\mathdefn{\genericGPsum}\coloneqq \GPsum[\polQ'][\pol']$, for some $\pol\normeq \pol'$ and $\polQ\normeq \polQ'$ with $\b E(\pol', \polQ') = \emptyset$.
\end{definition}

Different choices of $\pol', \polQ'$ with $\b E(\pol', \polQ') = \emptyset$ yield normally equivalent polytopes by \Cref{thm:EAdmissibilityConeIsDeformationCone}.
Moreover, as immediate consequence of \Cref{thm:EAdmissibilityConeIsDeformationCone}, we get the following description:

$$\c E(\pol, \polQ) \cong \scr L\bigl(\DefoCone[\genericGPsum]\bigr) \,\ssm\, \bigcup_{\substack{\pol'\deformed\pol,\, \polQ'\deformed\polQ \\ \pol'\not\normeq\pol \text{ or } \polQ'\not\normeq\polQ}} \scr L\bigl(\DefoCone[{\genericGPsum[\polQ'][\pol']}]\bigr).$$
	
Thus, $\c E(\pol,\polQ)$ is a graded lattice minus one of its ideals, see \Cref{fig:SC3_generic,fig:FaceLatticeVsEqualitySets}, and we get that:
\begin{compactenum}
\item[(a)] $\c E(\pol, \polQ)$ is a join-semi-lattice which is a sub-semi-lattice of $\scr L\bigl(\DefoCone[\genericGPsum]\bigr)$ (we actually proved the ``join'' part in \Cref{lem:EPQintersections});
\item[(b)] $\c E(\pol, \polQ)$ is co-graded, \ie for every $\b E\in\c E(\pol,\polQ)$, all saturated chains from $\b E$ to the maximum of $\c E(\pol,\polQ)$ (namely $\emptyset$) have the same length;
\item[(c)] $\c E(\pol, \polQ)$ is co-atomistic, \ie every element can be described as the meet of a subsets of co-atoms (a co-atom is an element covered by the maximum, namely here $\emptyset$).
\item[(d)] $\c E(\pol, \polQ)$ is the join-semi-lattice of ``closed faces'' (\ie the set of maximizers of a linear functional, when it exists) of the half-open cone:
$$\interior(\DefoCone\times\DefoCone[\polQ]) \cap \bigl\{ (\pol, \polQ) ~;~ \pol,\polQ\in \SC, \text{ with } \b E(\pol, \polQ) \supseteq \b E\bigr\};$$
where ``$\interior$'' denotes the relative interior of a cone, see \Cref{fig:SC3_generic} (right).
\end{compactenum}
	
In particular, we can partition $\scr L(\SC[n+1])$ (and more generally any $\scr L(\DefoCone[\genericGPsum])$) into join-semi-lattices of equality sets $\c E(\pol, \polQ)$, each of them having a unique maximum $\DefoCone[\genericGPsum]$.
Put together, these join-semi-lattices form a partition of $\scr L(\DefoCone[\genericGPsum])$, see \Cref{fig:SC3_generic,fig:FaceLatticeVsEqualitySets}.

Moreover, the collection of all these maxima is itself well-behaved inside $\scr L(\SC[n+1])$.
Recall that, in a cone $\polC$, the \defn{face figure} associated to the face $\polF$ of $\polC$ is the sublattice of $\scr L(\polC)$ formed by all the faces $\pol[G]$ of $\polC$ satisfying $\polF \subseteq \pol[G]$.

\begin{corollary}\label{cor:CubeIsInsideGenericGPsums}
The set $\bigl\{\DefoCone[\lift\!\genericGPsum] ~;~ \pol, \polQ\in \SC^2\bigr\}$ is the face figure of $\DefoCone[{\lift\!{([0, 1]^n)}}]$ inside $\scr L(\SC[n+1])$.
In particular, this defines a lattice isomorphism between $\scr L(\SC)\times\scr L(\SC)$ and the face figure of $\DefoCone[{\lift\!{([0, 1]^n)}}]$ inside $\scr L(\SC[n+1])$.
\end{corollary}

\begin{proof}
	Note that $[0, 1]^n = \GPsum[\b 0][\b 1]$ is the GP-sum for $(\b 0, \b 1)$ with equality set $\b E(\b 0, \b 1) = \emptyset$.
	By \Cref{thm:EAdmissibilityConeIsDeformationCone}, the cone $\DefoCone[{\GPsum[\b 0][\b 1]}]$ is a face of $\DefoCone[\GPsum]$ if and only if $\b 0\deformed\polQ$, and $\b 1\deformed\pol$ and $\emptyset\supseteq\b E$. 
    Note that the two first conditions hold for any $\pol,\polQ$, but the last condition holds if and only if $\b E = \emptyset$.
	Consequently, the faces of $\SC[n+1]$ containing $\DefoCone[{\lift\!{\!([0, 1]^n)}}]$  are precisely the faces of the form $\DefoCone[\lift\!{\genericGPsum}]$ for some $\pol,\polQ\in \SC^2$, ordered with the product order  $\genericGPsum\deformed\genericGPsum[\polQ'][\pol']$ if and only if $\polQ\deformed\polQ'$ and $\pol\deformed\pol'$. Thus, $\bigl\{\DefoCone[\lift\!\genericGPsum] ~;~ \pol, \polQ\in \SC^2\bigr\}$ is isomorphic to $\scr L(\SC)\times\scr L(\SC)$.    
\end{proof}

\begin{remark}
We have $\GPsum[\b 0][\b 1] = [0, 1]^n$, and $\lift\!{([0, 1]^n)} = \prod_{i=1}^n [\b0, \b e_i - \b e_{n+1}]$.
The cone $\DefoCone[{\GPsum[\b 0][\b 1]}]$ is the deformation cone of a parallelotope: it is simplicial (see \cite{CastilloDoolittleGoecknerRossYing-MinkowskiSummandsCubes,PadrolPilaudPoullot-DeformedGraphicalZonotopes, PadrolPoullot-NewCriteriaIndecomposability}).
\end{remark}

\begin{remark}
For $\pol\in \SC$, it is immediate to see that $\c E(\pol, \pol)$ is isomorphic to the interval inside the lattice $\scr L(\SC[n+1])$ between $\DefoCone[{\lift\!(\GPsum[\pol][\pol])}] \simeq \DefoCone$ and the generic GP-sum $\DefoCone[{\lift\!(\genericGPsum[\pol][\pol])}]$.
In particular, $\c E(\pol[\Pi]_n, \pol[\Pi]_n)$ is isomorphic to the face figure of $\DefoCone[{\lift{\!\pol[\Pi]_n}}]$ inside $\scr L(\SC[n+1])$, see \Cref{fig:FaceLatticeVsEqualitySets}.
\end{remark}

\subsection{Dimension of deformation cones of generic GP-sums}

\begin{theorem}\label{thm:dimCadmissibility}
Let $\pol, \polQ\in\SC$, then:
$\dim\DefoCone[\lift\!(\genericGPsum)] = \dim\DefoCone + \dim\DefoCone[\polQ] + 1$.
\end{theorem}

\begin{proof}
Lifts add $1$ to the dimension.
By \Cref{prop:AdmissConeIsPolyhedral},
$\DefoCone[\genericGPsum]  = \admisscone$ is isomorphic to a subset of $\DefoCone \times \DefoCone[\polQ]$.
Hence, $\dim\admisscone \leq \dim\DefoCone + \dim\DefoCone[\polQ]$.
Conversely, fix $\pol,\polQ\in\SC$ with $\b E(\pol,\polQ) = \emptyset$.
Having an empty equality set amounts to satisfying some inequalities \emph{strictly}.
For any $\pol'\in \DefoCone$ in an arbitrarily small ball in $\DefoCone$ centered at $\pol$ (of dimension $\dim\DefoCone$), and respectively for~$\polQ'$, the same inequalities are satisfied strictly, hence $\b E(\pol', \polQ') = \emptyset$.
As $\pol'\sim\pol$, and $\polQ'\sim\polQ$, and $\b E(\pol',\polQ') = \b E(\pol, \polQ)$, \Cref{thm:EAdmissibilityConeIsDeformationCone} gives: $\GPsum[\polQ'][\pol'] = \genericGPsum[\polQ'][\pol']\in\interior\bigl(\admisscone\bigr)$.
Thus, we get: $\dim\admisscone\geq \dim\DefoCone + \dim\DefoCone[\polQ]$. %
\end{proof}

\begin{remark}
One also gets:
$\dim\DefoCone[\GPsum] = \dim\DefoCone + \dim\DefoCone[\polQ] - \corank_{\c E(\pol,\polQ)}\b E$, by \Cref{thm:EAdmissibilityConeIsDeformationCone}, where $\mathdefn{\corank_{\c E(\pol,\polQ)}\b E}$ denotes the co-rank of $\b E$ in the join-semi-lattice $\c E(\pol,\polQ)$ (where $\corank_{\c E(\pol,\polQ)}\emptyset = 0$).
Since $\c E(\pol, \polQ)$ is hard to determine, we do not use this general formula.
\end{remark}

Besides, the equality set formulation allows to access the facets of generic faces of $\SC[n+1]$.
We denote $\mathdefn{\facetnumber\pol[C]}$ the number of facets of the cone $\pol[C]$, and $\mathdefn{\coatomnumber\c E(\pol,\polQ)}$ the number of elements covered by $\emptyset$ in $\c E(\pol,\polQ)$.
For $\pol, \polQ\in \SC$, one can prove that:
$\coatomnumber\c E(\pol,\polQ) \geq n$, and that
$\facetnumber\DefoCone[\lift\!(\genericGPsum)] \geq \facetnumber\DefoCone \,+\, \facetnumber\DefoCone[\polQ] \,+\, \coatomnumber\c E(\pol,\polQ)$.

We now use generic GP-sums to bound the number of faces of $\SC$, and leave as an open question to explore these properties of non-generic GP-sums in order to improve these bounds.

\subsection{Estimating the $f$-vector of the submodular cone}\label{sec:NumbersAndBoundsOfFaces}

We extend our analysis of \cref{ssec:NewRays} to bound the number of faces of the submodular cone.

We denote $\mathdefn{f(n, k)}$ the number of $k$-faces of $\SC$, and let $\mathdefn{F_n(X)} \coloneqq \sum_{k\geq n} f(n, k)\, X^k$.
Let $\mathdefn{q_n} \coloneqq F_n(1)$ be the total number of faces of $\SC$.
A quick computer experiment gives:
\begin{align*}
F_3(X) =&~ X^3 + 5X^4 + 9 X^5 + 6 X^6 + X^7,&\text{ and }\\
F_4(X) =&~ X^4 + 37 X^5 + 356 X^6 + 1596 X^7 + 3985 X^8 + 5980 X^9 + 5560 X^{10} \\&+ 3212 X^{11} + 1128 X^{12} + 228 X^{13} + 24 X^{14} + X^{15}
\end{align*}

For polynomials $P = \sum_k p_kX^k$, and $Q = \sum_k q_kX^k$, we write $\mathdefn{P \succeq Q}$ to mean $p_k \geq q_k$ for all $k$.

\begin{corollary}\label{cor:FpolynomialSCn}
For all $n$, we have:
$F_{n+1}(X) \succeq X\cdot F_n(X)^2$. Hence: $F_n(X) \succeq X^{2^{n-k}-1} \cdot F_k^{2^{n-k}}$, for all $k < n$.
Moreover, the total number of faces $q_n$ of $\SC$ satisfies $q_n \geq 22\, 108^{2^{n-4}}$ for $n\geq 4$.
\end{corollary}

\begin{proof}
By \Cref{cor:CubeIsInsideGenericGPsums,thm:dimCadmissibility}, there is an injection $\scr L(\SC)^2 \hookrightarrow \scr L(\SC[n+1])$, such that the dimension of the face associated to $(\polF, \pol[G]) \in \scr L(\SC)^2$ is $\dim\polF + \dim\pol[G] + 1$.
This implies that for all $n, a$, we have: $f(n+1, a) \geq \sum_{i+j = a-1} f(n, i)\cdot f(n, j)$.
This gives the first claim, and the second claim follows.
The third claim comes from $q_n = F_n(1) \geq F_4(1)^{2^{n-4}}$, where $F_4(1) = 22\,108$.
\end{proof}

\begin{remark}
Note that the highest dimensional face of $\SC$ has dimension $2^n - 1$.
The previous corollary fails at grasping the number of low dimensional faces.
For instance, the rays of $\SC$ are faces of dimension $n+1$.
However, the non-zero coefficient of lowest degree in $F_4^{2^{n-4}}$ is of degree $2^{n-2}$, so \Cref{cor:FpolynomialSCn}, given $F_4$, yields a non-trivial lower bound only for $f(n, k)$ with $k\geq 2^{n-2}$.
\end{remark}

Recall that a face of $\SC$ is simplicial if and only if it is a cone over a simplex, of equivalently if its number of facets equals its dimension minus its lineality.
We can lower bound the number of non-simplicial faces using the following lemma:

\begin{lemma}\label{lem:DCPQwelcomesDCP}
Let $\pol, \polQ\in \SC$, then $\scr L(\DefoCone)$ is isomorphic to the interval of $\DefoCone[\genericGPsum]$ between $\DefoCone[{\genericGPsum[\polQ'][\b 1]}]$ and $\DefoCone[{\genericGPsum[\polQ'][\pol]}]$ for any $\polQ'\in \DefoCone[\polQ]$.
\end{lemma}

\begin{proof}
Follows from \Cref{thm:EAdmissibilityConeIsDeformationCone}, because $\emptyset\in \c E(\pol', \polQ')$ for any $\pol',\polQ'\in \SC$ by \Cref{rmk:GenericGPsum}.
\end{proof}

\begin{remark}
Note that, in general, a face $\polF$ of a polytope $\pol$ is \textbf{not} a deformation of $\pol$, hence it is \textbf{not} true that $\DefoCone[\polF]\subseteq\DefoCone[\pol]$.
The previous lemma is really specific to our situation.
\end{remark}

\begin{corollary}\label{cor:NumberNonSimplicialFaces}
The number of faces of $\SC$ that are not simplicial is greater than $\prod_{2\leq m < n} q_m$.
\end{corollary}

\begin{proof}
By \Cref{lem:DCPQwelcomesDCP}, if $\DefoCone$ is a non-simplicial face of $\SC$, then $\DefoCone[\genericGPsum]$ and $\DefoCone[{\lift\!(\genericGPsum)}]$ are also not simplicial (because all intervals of Boolean lattices are Boolean).
Hence, denoting $w_n$ the number of non-simplicial faces of $\SC$, we have $w_{n+1} \geq w_n \cdot q_n$, yielding $w_n\geq \prod_{2\leq m < n} q_m$.
\end{proof}

In the Upper Bound \Cref{thm:UBTgeneral}, the number of facets is maximized by neighborly polytopes, which are simplicial polytopes\footnote{Not to be confused with a simplicial cone, which is a cone over a simplex.}. %
\Cref{cor:NumberNonSimplicialFaces} states that $\SC$ is ``far'' from being (a cone over) a simplicial polytope: it has ``a lot'' of faces which are not (a cone over) simplices.
We leave as an open question to quantitatively improve the upper bound in \Cref{thm:UpperLowerBoundsTnSn} using this discrepancy.

Yet, some faces of $\SC$ are simplicial (implying that faces of these faces are too), 
notably the face associated to Loday's associahedron \cite{ArkaniHamedBaiSongYan-ScateringAmplitudes}, to the cube \cite{CastilloDoolittleGoecknerRossYing-MinkowskiSummandsCubes}, to some specific nestohedra \cite[Prop.~3.29 \& 3.32]{PadrolPilaudPoullot-DeformationConeNestohedra}, to graphical zonotopes for triangle-free graphs \cite[Cor.~2.10]{PadrolPilaudPoullot-DeformedGraphicalZonotopes}.
We think that these \emph{uniquely decomposable} instances (c.f.~\cite{PadrolPoullot-NewCriteriaIndecomposability}) should be regarded as rare and intriguing.
We conjecture:

\begin{conjecture}
When $n\to +\infty$, the number of simplicial faces of the submodular cone $\SC$ is negligible in comparison to the (total) number of faces of $\SC$.
\end{conjecture}

\bibliographystyle{alpha}
\bibliography{Biblio.bib}
\label{sec:biblio}

\end{document}